\documentclass{article}
\usepackage{graphicx}
\usepackage[utf8]{inputenc}
\usepackage{algorithmic}
\usepackage{algorithm}
\usepackage{amsmath}
\usepackage{amssymb}
\usepackage{amsthm}
\usepackage{xcolor}
\usepackage{bm}
\usepackage{enumitem}
\usepackage{braket}
\usepackage{mathrsfs}
\usepackage{listings}
\usepackage{hyperref}
\usepackage{soul}

\newcommand{\RR}{\mathbb{R}}

\newcommand{\ip}[1]{\langle #1 \rangle}

\newcommand{\f}{\mathbb}
\newcommand{\vf}{\varphi}

\newcommand{\wt}{\widetilde}

\newcommand{\tl}{\triangleleft}

\newcommand{\ms}[1]{\{#1^{(N)}\}_N}
\newcommand{\GLT}{\sim_{GLT}}

\newcommand{\acs}{\xrightarrow{a.c.s.}}

\newcommand{\ve}{\varepsilon}

\newcommand{\nn}{{\bm n}}
\newcommand{\ii}{{\bm i}}
\newcommand{\jj}{{\bm j}}

\newcommand{\kk}{{\bm k}}
\newcommand{\uu}{{\bm 1}}

\newcommand{\svs}{\sim_\sigma}
\newcommand{\btheta}{\boldsymbol\theta}

\DeclareMathOperator{\rk}{rk}

                \newcommand{\bfx}{\mathbf x}    

\providecommand{\AC[1]}{{\color{red} #1}}
\providecommand{\GB[1]}{{\color{blue} #1}}

\usepackage[top=2.00cm,bottom=3.00cm,left=2.00cm,right=2.00cm]{geometry}

\theoremstyle{definition}
\newtheorem{definition}{Definition}
\theoremstyle{plain}
\newtheorem{theorem}{Theorem}

\newtheorem{lemma}{Lemma}
\newtheorem{remark}{Remark}

\title{Extension and convergence analysis of Iterative Filtering to~spherical~data}

\author{Giovanni Barbarino\thanks{University of Mons, Mons, Walloon Region, Belgium},Roberto Cavassi\thanks{University of L'Aquila, L'Aquila, Italy},Antonio Cicone\thanks{University of L'Aquila, L'Aquila, Italy, and Istituto di Astrofisica e Planetologia Spaziali, INAF, Rome, Italy, and Istituto Nazionale di Geofisica e Vulcanologia, Rome, Italy}}

\begin{document}

\maketitle

\begin{abstract}
 Many real-life signals are defined on spherical domains, in particular in geophysics and physics applications. In this work, we tackle the problem of extending the iterative filtering algorithm, developed for the decomposition of non-stationary signals defined in Euclidean spaces, to spherical domains. We review the properties of the classical Iterative Filtering method, present its extension, and study its convergence in the discrete setting. In particular, by leveraging the Generalized Locally Toeplitz sequence theory, we are able to characterize spectrally the operators associated with the spherical extension of Iterative Filtering, and we show a counterexample of its convergence. Finally, we propose a convergent version, called Spherical Iterative Filtering, and present numerical results of its application to spherical data.
\end{abstract}

\section{Introduction}

Real-life signals are mainly non-stationary, i.e. their features change over time or space, and are produced by nonlinear phenomena. This implies that classical signal processing methods, like Fourier or Wavelet Transform, can prove to be limited for the analysis and decomposition of such signals. For this reason in the late 90s, a group of researchers at NASA headed by Norden Huang developed a completely new approach for non-stationary signal processing called Empirical Mode Decomposition (EMD).

The EMD method, the first of its kind, is based on the iterative computation of the signal moving average via envelopes connecting its extrema. The computation of the signal moving average allows to split the signal itself into a small number of simple and non-stationary oscillatory components, called Intrinsic Mode Functions (IMFs), which are separated in frequencies and almost uncorrelated \cite{huang1998empirical}.

EMD proved to be a really powerful method in many applied fields of research
\cite{abbasimehr2020optimized, cao2019financial, cui2021rolling, holmberg2017influence, stetco2019machine, zeng2020review, zhang2020abrupt}. However, it is prone to instability and the so-called mode-mixing \cite{wu2009ensemble}. For this reason, many EMD variants have been proposed over the years, like the Ensemble Empirical Mode Decomposition (EEMD) \cite{wu2009ensemble}, the complementary EEMD \cite{yeh2010complementary}, the complete EEMD \cite{torres2011complete}, the partly EEMD \cite{zheng2014partly}, the noise assisted multivariate EMD (NA-MEMD) \cite{ur2011filter} and many others. They all allow us to address the instability issue as well as to reduce the so-called mode mixing problem.
However, EMD and all these variants are still missing a rigorous mathematical analysis, due to the usage of a number of heuristic and ad hoc elements for the computation of the signal moving average. Some results have been presented in the literature \cite{huang2009convergence, huang2014introduction, ur2011filter}, but a complete analysis is still missing.

Given these limitations, but also the considerable attention garnered by these methodologies within the global scientific community, numerous research groups have embarked on investigations into this domain, proffering alternative approaches to signal decomposition. Noteworthy methodologies include the Variational mode decomposition \cite{dragomiretskiy2013variational}, sparse time-frequency representation \cite{hou2011adaptive, hou2009variant}, Blaschke decomposition \cite{coifman2017carrier}, Geometric mode decomposition \cite{yu2018geometric},  Empirical wavelet transform \cite{gilles2013empirical}, and analogous techniques \cite{meignen2007new, pustelnik2012multicomponent, selesnick2011resonance}. All these techniques hinge on optimization relative to a predetermined basis.

The sole alternative method presented in existing literature founded on iterative principles and thus obviating the need for a priori assumptions about the signal under study is the Iterative Filtering (IF) algorithm \cite{lin2009iterative}. We recall here also its fast implementation via FFT, denoted as Fast Iterative Filtering (FIF) \cite{cicone2020numerical}, along with its extensions, namely Adaptive Local Iterative Filtering (ALIF) \cite{cicone2016adaptive} and Resampled Iterative Filtering (RIF) algorithms \cite{ barbarino2021stabilization, barbarino2022conjectures}, which are designed to handle signals characterized by pronounced non-stationarities, such as chirps, whistles, and multipath. Despite their recent publication, these iterative methods have found effective application across diverse domains \cite{ghobadi2020disentangling, li2018entropy, materassi2019stepping, mitiche2018classification, papini2020multidimensional, piersanti2020inquiry, piersanti2020magnetospheric, sharma2017automatic,  spogli2019role, spogli2021adaptive,  yu2010modeling}.

The structural framework of the IF, ALIF, and RIF algorithms mirrors that of EMD, with the primary divergence lying in the computation of the signal moving average. Unlike EMD, these methods compute the signal moving average through convolution with a preselected filter function, as opposed to employing the mean between two envelopes. This seemingly straightforward distinction has paved the way for a comprehensive mathematical analysis of IF, ALIF, and RIF. Notably, \cite{cicone2020numerical} and \cite{huang2009convergence} delve into the mathematical analysis of the IF algorithm, identifying a priori conditions for its convergence. In \cite{cicone2020study}, the examination of boundary effects in the IF algorithm is conducted, presenting a formula to preemptively estimate the extent of error propagation within each component of a decomposition. \cite{stallone2020new} introduces a methodology to mitigate boundary effects across decomposition algorithms, including IF, and investigates IF's efficacy in segregating components originating from stochastic processes. Building on the analyses conducted for EMD \cite{rilling2007one} and Synchrosqueezing \cite{wu2011one}, \cite{cicone2024one} scrutinizes IF's ability to disentangle two stationary frequencies within a signal. Lastly, \cite{barbarino2021stabilization}, \cite{barbarino2022conjectures}, and \cite{cicone2019spectral} delve into the convergence of the ALIF method and propose the alternative, faster, and convergent RIF algorithm.

EMD and IF methods have been extended to deal with higher dimensional \cite{ cicone2017multidimensional, rilling2007bivariate, tanaka2007complex, thirumalaisamy2018fast} and multivariate signals \cite{cicone2022multivariate, rehman2010multivariate} defined in a Euclidean space. Regarding the extension of this method to non-Euclidean spaces, and in particular, to the case of a sphere, a version of EMD has been proposed in the paper \cite{fauchereau2008empirical}, whereas IF has never been extended yet. However, there are many interesting real-life applications in which the data are sampled on a spherical domain. We can think, for instance, of measurements regarding the Earth, like the surface temperature \cite{yun2019new} and pressure \cite{sweeney2021products}, or geophysical quantities measured via satellites, like the Earth's magnetic and electric field as measured by the ESA \cite{friis2006swarm}, NOAA \cite{loto2019goes} and CNSA \cite{cao2018electromagnetic} missions, and the Earth's gravitational field \cite{marsh1988new}, or astrophysical measurements, like the cosmic microwave background \cite{kamionkowski1999cosmic}.
All these kinds of data have been studied so far using classical linear approaches, like spherical harmonic analysis and Fourier transform in higher dimension \cite{hu2002cosmic, klosko1982spherical, pierret2013optimal, thebault2021spherical}. There is thus a need to develop new algorithms able to handle the nonstationarities contained in these signals.

In this work, we propose an extension of IF to spherical data and study the a prior convergence of this method in the discrete setting.
The work is organized as follows: in Section \ref{sec:IF} we review IF and introduce its generalization to the sphere. In Section \ref{sec:GLT} the Generalized Locally Toeplitz (GLT) theory is recalled, in particular, for the case of 2-level GLT sequences, and used in order to study the convergence properties of the proposed method. Section \ref{sec:Num_Ex} reports a convergent version of Spherical Iterative Filtering (SIF) and numerical examples of its application. The paper ends with conclusions and future research directions.

\section{Iterative Filtering and its adaptation to the sphere}\label{sec:IF}

In order to have a better understanding of how we adapted Iterative Filtering on the sphere, it is useful to have a brief description of the original algorithm.

\subsection{How Iterative Filtering works}
As stated in the introduction, the key difference between EMD and Iterative Filtering (IF) is the computation of the local average.
As stated in the name, in IF methods the computation of the local average is made through filtering, which means convolution with a function called a filter.
\begin{definition}\label{def:filter} $ $
\begin{itemize}
    \item [(1)] A function $w:[-l, l]\to \mathbb{R}$ is a filter if it is nonnegative, even, bounded, continuous, and $\int_{\mathbb{R}} w(t)\,dt = 1$.
    \item[(2)] A double convolution filter $w$ is the self-convolution of a filter $\tilde{w}$, that is $w=\tilde{w}*\tilde{w}$.
    \item[(3)] The size, or the length, of a filter $w$ is half the measure of its support.
\end{itemize}
\end{definition}

Before describing the pseudo-code of Iterative Filtering, it is necessary to define the Sifting operator. This operator will be applied iteratively in the IF methods in order to extract the Intrinsic Mode Functions (IMF).

\begin{definition}[\textbf{Sifting operator}]\label{sifting}
    Let $f \in L^2(\mathbb{R})$  be a filter. The moving average of a signal
    $g\in\RR^n$ can be computed as
\begin{equation}\label{eq:moving_average_IF}
    \mathcal L(g)(x) = \int_{\mathbb{R}}g(t)f(x-t)dt.
\end{equation}
    The associated sifting operator is $\mathcal M:L^2(\mathbb{R})\to L^2(\mathbb{R})$ s.t.
\begin{equation}
    \mathcal M(g)(x):= (g-f*g)(x) = g(x) - \mathcal L(g)(x)
\end{equation}
\end{definition}

In practical applications we always study the signal $g(x)$ on an interval, say $[0,1]$. Outside this interval, the signal is usually not known, so we have to impose some boundary conditions, discussed for example in \cite{cicone2020study,stallone2020new}. In particular, in \cite{stallone2020new} the authors show how any signal can be pre-extended and made periodic at the boundaries, for example, by reflecting the signal on both sides and making it decay. Therefore, for simplicity and without losing generality, we will assume that the signals to be decomposed are $1$-periodic.

 In a discrete setting, a signal is usually given as a vector of sampled values $\bm g = [g_i]_{i=0,\dots,n-1}$ where $g_i = g(x_i)$ and $x_i = i/n$ for $i\in \mathbb Z$. As a consequence, one can discretize the IF moving average \eqref{eq:moving_average_IF} with a simple quadrature formula
\begin{equation}\label{eq:moving_average_dIF}
	\mathcal L(g)(x_i) = \int_{\mathbb R} g(z)f(x_i-z) \,\,{{\rm d}}z \approx \frac 1{nK}\sum_{j\in\mathbb Z} g_j f(x_i-x_j).
\end{equation}
Since the filter $f(z)$ has compact support of size $2L$, the above formula is always well-defined. Here $K$ is a normalizing constant depending on $f(z)$ and $n$ defined as
\[
K : = \frac 1{n}\sum_{j\in\mathbb Z} f(x_i-x_j) \approx  \int_{\mathbb R} f(z) \,\,{{\rm d}}z  = 1
\]
ensuring that the quadrature formula actually performs a local convex combination of the signal points $g_i$, akin to the averaging operation performed by the convolution in the continuous case.

Equation \eqref{eq:moving_average_dIF} can be expressed through a Hermitian circulant matrix $F$ with first row
\[\bm f_1 :=
\frac{1}{nLK}
\left[
f(0), \, f\left( \frac{1}{nL}  \right)
, \dots
, \, f\left( \frac{s}{nL}  \right)
, \, 0
\dots
\, 0
, \, f\left( \frac{s}{nL}  \right)
, \dots
, \, f\left( \frac{1}{nL}  \right)
\right]
\]
where $s = \lfloor nL \rfloor < \lfloor n/6 \rfloor$.
The sampling vector of $\mathcal M(g)$ on the points $x_0,x_1,\dots, x_{n-1}$, that we indicate as $\mathcal M(\bm g)$, is thus rewritten as a matrix-vector multiplication
\begin{equation}\label{eq:matrix_K_IF}
	\mathcal M(g)(x_i) = g_i - \mathcal L(g)(x_i)
	\implies
	\mathcal M(\bm g) = (I - F)\bm g.
\end{equation}

\begin{theorem}[\cite{cicone2020numerical}]\label{thm:FIF_}
Given a signal $\bm g\in\RR^n$, assuming that we are considering a double convolution filter, the iterated application of the Sifting operator to the sampling of the signal converges to
\begin{equation}\label{eq:discreteIMF}
\textrm{IMF}_1=\lim_{m\rightarrow \infty}(I-F)^m \bm g=U Z U^T \bm g
\end{equation}
where $Z$ is the diagonal binary matrix containing in the diagonal an entry equal to $1$ in correspondence to each zero component of the discrete Fourier Transform of the filter $f$, and $U$ is the orthogonal matrix containing, as columns, the Fourier basis vectors.
\end{theorem}

This theorem guarantees the a priori convergence of the IF algorithm for any possible signal $g$. The pseudocode of the discrete Iterative Filtering is reported in Algorithm \ref{alg:DIF}.

\begin{algorithm}
\caption{\textbf{Discrete Iterative Filtering} IMF = DIF$(\bm g)$}\label{alg:DIF}
\begin{algorithmic}
\STATE IMF = $\left\{\right\}$
\WHILE{the number of extrema of $\bm g$ $\geq 2$}
      \STATE $\bm g_1 = \bm g$
      \WHILE{the stopping criterion is not satisfied}
                  \STATE  compute the filter length $l_m$ for $\bm g_{m}(x)$ and generate the corresponding convolution matrix $F_m$
                  \STATE  $\bm g_{m+1}(x) = (I-F_m) \bm g_{m}(x)$
                  \STATE  $m = m+1$
      \ENDWHILE
      \STATE IMF = IMF$\,\cup\,  \{\bm g_{m}\}$
      \STATE $\bm g=\bm g-\bm g_{m}$
\ENDWHILE
\STATE IMF = IMF$\,\cup\,  \{ \bm g\}$
\end{algorithmic}
\end{algorithm}

The algorithm would run infinitely many times. To have finite time computations it is possible to use a stopping criterion. In the literature, the standard stopping criterion used \cite{lin2009iterative} is the relative error in norm 2
\begin{equation}\label{eq:stopping_criterion}
\frac{\|\mathcal{F}_m^{p+1}(\bm g_m)-\mathcal{M}_m^{p}(\bm g_m)\|_2}{\|\mathcal{F}_m^{p}(\bm g_m)\|_2}\leq \delta.
\end{equation}
for a prefixed $\delta>0$.

In the following, we want to extend these results to the case of the sphere.

\subsection{The spherical setting}

\subsubsection{A useful operator}
The topology of $\mathbb{S}^2$ is the main difference from the classic setting of Iterative Filtering. In fact, since the sphere is not diffeomorphic to a portion of $\mathbb{R}^2$, at least two charts are necessary to make an atlas, as a consequence, it is impossible to extend the convolution as it is defined on $\mathbb{R}^n$.

In the literature, many operators have been defined on the sphere in order to adapt some of the properties of the convolution. In particular, the most interesting one that fits better our necessities is the one called isotropic convolution.
Before its definition, it is necessary to describe some properties of $SO(3)$, the rotation group of the sphere.

Given $f\in \mathcal{H}^0(\mathbb{S}^2)$, where ${H}^0(\mathbb{S}^2):=\{f\in L^2(\mathbb{S}^2): f(\theta,\phi) = f(\theta)\}$ which is the subset of functions in $L^2(\mathbb{S}^2)$ with azimuthal/rotational symmetry about the north pole axis, it is easy to see that $\forall \rho \in SO(3), \exists z_0 \in \mathbb{S}^2$ s.t.
$$
    \mathcal{R}_{\rho}f(z)=\mathcal{R}_{z_0}f(z),
$$
where $\mathcal{R}_{z_0}(f)$ can be seen as a translation of the axis of $f$, that now is the one passing through $z_0$.
That said, it is possible to introduce the following notation:
$$
    f_{z_0}:=\mathcal{R}_{z_0}f
$$
and the convention that if $f \in \mathcal{H}^0(\mathbb{S}^2)$, then $f_0:=f$.



Now we are ready to define the following.

\begin{definition}[\textbf{Isotropic convolution on the sphere}]\label{isoconv}
    Let $g \in L^2(\mathbb{S}^2), f\in \mathcal{H}^0(\mathbb{S}^2)$, their isotropic convolution $g \circledast f$ is defined as:
$$
   g \circledast f (z) = \int_{\mathbb{S}^2}g(\omega)f_z(\omega)d\omega
$$

\end{definition}

It is useful to notice some interesting properties of this operator. First of all, the isotropic convolution is very different from a convolution, mainly for the strong requests on the function $f$.
Another observation is that the resulting function is defined on $\mathbb{S}^2$ thanks to the property mentioned above of the action of $SO(3)$ for functions in $\mathcal{H}^0(\mathbb{S}^2)$. If $f$ was a generic function in $L^2(\mathbb{S}^2)$, the resulting function would have been defined on $SO(3)$.
Finally, we present another definition of isotropic convolution that can be seen in literature \cite{iso-conv}, but it is clear that the meaning remains the same.
\begin{definition}[\textbf{Isotropic convolution on the sphere}]\label{isoconv2}
    Let $g \in L^2(\mathbb{S}^2), F\in L^2([-1,1])$, their isotropic convolution $g \circledast F$ is defined as:
$$
   g \circledast F (z) = \int_{\mathbb{S}^2}g(\omega)F(z \cdot \omega)d\omega
$$
\end{definition}

It is easy to prove that Definition \ref{isoconv} and Definition \ref{isoconv2} are equivalent when $F(n\cdot w) = f(w)$ with $n$ be the north pole. In fact, notice that changing the vector $n$ to $z$  corresponds to a rotation of the function $f$.
This alternative definition is interesting because it makes clearer the role of $F$ as a kernel: it modulates the values of $g$, depending on the distance from a certain point $z$.

\subsubsection{Adaptation of Iterative Filtering to the sphere}

The isotropic convolution allows us to adapt Iterative Filtering to the spherical case.
Taking $f$ with connected support, non-negative, with $||f||_{L^1}=1$ and, as requested, axially symmetric, then it can be seen as a weight function that we can use to evaluate the local average around the north pole. Hence $\int_{\mathbb{S}^2}g(\omega)f_z(\omega)d\omega$ can be seen as the evaluation of the local average of the function $g$ in $z$, using $f$ as a weight function on this manifold. From now on, a function $f$ is a filter if it possesses all the above properties.

Now it is possible to define the main operator.

\begin{definition}[\textbf{Sifting operator on the sphere}]\label{siftingS}
    Let $f \in \mathcal{H}^0(\mathbb{S}^2)$  be a filter. Its associated sifting operator is $\mathcal M:L^2(\mathbb{S}^2) \to L^2(\mathbb{S}^2)$ defined as    $$
    \mathcal M(g)(z):=(g- g \circledast f)(z) = g(z) - \int_{\mathbb{S}^2}g(\omega)f_z(\omega)d\omega
    $$

\end{definition}

It is easy to show that this operator is linear and bounded. In order to have the convergence of the limit
\begin{equation}\label{limitM}
  \lim_{n\to \infty, n\in \mathbb{N}} \mathcal M^n(g)
\end{equation}
$\forall g \in L^2(\mathbb{S}^2)$,  it is necessary to prove that $||\mathcal M||_{op}\leq 1$.

Since we want to study the discretized version of this operator, we have to study if the spectral radius $\rho(\mathcal M)\leq 1$ for the sifting operator obtained from a certain filter function $f$.
Since the only eigenvalue of the identity is $1$, it is necessary to study the discretized version of the operator $\cdot \circledast f$ and see when its eigenvalues are in the complex circle centered in $1$ with radius $1$.

To proceed, it is necessary to define the discretization grid and methods: the elements of the mesh grid are called $\{z_j\}_j$, and their associated partition of the sphere is $\{S_j\}_j$, such that $\bigcup_j S_j=\mathbb{S}^2$ and $S_i \cap S_j = \emptyset$ if $i \neq j$.

Now, the isotropic convolution can be written as:
$$
g \circledast f (z) = \int_{\mathbb{S}^2}g(\omega)f_{z}(\omega)d\omega = \sum_j \int_{S_j} g(\omega) f_{z}(\omega)d\omega
$$
If we study this operator only for $g$ in the subspace of $L^2(\mathbb{S}^2)$ generated by $\{ \chi_{S_j} \}_j$, it is possible to define the matrix associated with this operator as:

\begin{equation}\label{eq:sifting_matrix}
    B_{k,j}:=\frac{1}{\sigma(S_k)}\int_{S_k}\int_{S_j} f_{\omega}(r)d\sigma (r) d\sigma (\omega)
\end{equation}

Where $\sigma(S_k)$ is the measure of $S_k$ as a set on the surface of the unitary sphere.
It is interesting to notice that the normalisation and the first integral are necessary in order to evaluate the average of the resulting smoothed function and to project it in $\text{span}(\{ \chi_{S_j} \}_j)$.
On the other hand, the second integral is necessary to describe the isotropic convolution.
It is easy to prove that $B$ is stochastic since its elements are non-negative and the sum of the elements on one row is

$$
\sum_j \frac{1}{\sigma(S_k)}\int_{S_k}\int_{S_j} f_{\omega}(r)d\sigma (r) d\sigma (\omega) = \frac{1}{\sigma(S_k)}\int_{S_k}\int_{\mathbb{S}^2}f_{\omega}(r)d\sigma (r) d\sigma (\omega) = \frac{1}{\sigma(S_k)}\int_{S_k}||f_{\omega}||_{L^1}d\sigma (\omega) = 1
$$

An interesting property of this class of matrix is that its spectral radius is 1. Hence, it is necessary that the eigenvalues of $B$ are in the intersection between the complex disk centred in $0$ with radius $1$ and the one centred in $1$ with the same radius (as stated above), in order to obtain the convergence of the limit \eqref{limitM}.

\subsubsection{Case of conic filters}

Since the positivity of the eigenvalues is not granted, we choose a certain filter $f_0$ and study the behaviour of its eigenvalues when discretised on a certain family of mesh grids on the sphere.
The following function is called the truncated cone and it can be shown to be a filter.
\begin{equation}
    f_0(\theta,\phi) = \frac{(r-\phi)^+}{C} = \frac{(r-d((0,0),(\theta,\phi)))^+}{C} \quad \textrm{and} \quad f_{(\theta',\phi')}(\theta,\phi) = \frac{(r-d((\theta',\phi'),(\theta,\phi)))^+}{C}
\end{equation}

Where $r$ is the radius of the filter, $C$ a normalizing constant and $d$ is the spherical arc distance between two points.
The grid we decided to study is similar to a Driscoll and Healy mesh grid: $N^2$ points equally spaced in longitude and latitude, excluding the poles.
Later in the paper, we will give a more in-depth description of the coordinates of those points.
We decided to use this type of discretization since it is close to the way satellite data are collected.

\section{GLT Analysis and convergent version of the algorithm}\label{sec:GLT}

\subsection{Theory of 2-level Generalized Locally Toeplitz Sequences}	
Here we recall the basic notions, results and concepts of $2$-level GLT sequences and linked subjects, without going too much into technical details. All the results we report in this section can be found more in detail in \cite[Chapter 6]{GLT-bookII}, altogether with an extensive and complete discussion about the GLT sequences and an extension to $d$-level matrix sequences.
	
	
When dealing with multilevel sequences, matrices and vectors, we will use the multi-index notation.
A multi-index $\ii\in\mathbb N^2$ is simply a vector in $\mathbb N^2$; its components are denoted by $i_1,i_2$.
If we write $X=[x_{\ii\jj}]_{\ii,\jj=\uu}^{\nn}$, then $X$ is a $N^2\times N^2$ matrix whose components are indexed by two $2$-indices $\ii,\jj$, both varying from $\uu = (1,1)$ to $\nn=(N,N)$ according to the lexicographic ordering.	
	In this context, by a sequence of matrices (or matrix-sequence) we mean a sequence of the form $\{A^{(N)}\}_N$ of dimension $N^2\times N^2$. The entries of the matrix $A^{(N)}$ will be indexed by two $2$-indices $\ii=(i_1,i_2)$, $\jj=(j_1,j_2)$, where $1\le i_1,i_2,j_1,j_2\le N$.

\subsubsection{Spectral Symbol and Zero-Distributed Sequences}

 We say that a matrix-sequence $\{A_n\}_n$ admits a spectral symbol (or for short, just symbol) $f(x)$ when the sampling of $f(x)$ on an uniform $n$-grid over its domain yields an approximation of the eigenvalues of $A_n$, that gets better as $n\to \infty$.
In a sense, the symbol encodes the asymptotic behavior of the spectrum for the whole sequence, thus proving to be an invaluable tool for the analysis of iterative methods that make use of the sequence.
The same concept can be expressed also for the singular values of teh matrices in the sequence, and the rigorous formal definition for the symbol is based of an ergodic formula that must hold for every test functions $F\in C_{c}(\mathbb C)$ ($C_{c}(\mathbb R)$) as follows.
	
\begin{definition}[\textbf{asymptotic singular value and eigenvalue distributions of a matrix-sequence}]\label{dd}
Let $\{A^{(N)}\}_N$ be a matrix-sequence with $A^{(N)}$ of size $N^2\times N^2$, and let $f:\Omega\subset\mathbb R^d\to\mathbb C$ be measurable with $0<\mu_d(\Omega)<\infty$.
\begin{itemize}[nolistsep,leftmargin=*]
	\item We say that $\{A^{(N)}\}_N$ has an asymptotic eigenvalue (or spectral) distribution described by $f$ if
	\begin{equation}\label{sd}
\lim_{n\to\infty}\frac1{N^2}\sum_{i=1}^{N^2}F(\lambda_i(A^{(N)}))=\frac1{\mu_d(\Omega)}\int_\Omega F(f(\bm x))\mathrm{d}\bm x,\qquad\forall\,F\in C_c(\mathbb C).
	\end{equation}
	In this case, $f$ is called the eigenvalue (or spectral) symbol of $\{A^{(N)}\}_N$ and we write $\{A^{(N)}\}_N\sim_\lambda f$.
	\item We say that $\{A^{(N)}\}_N$ has an asymptotic singular value distribution described by $f$ if
	\begin{equation}\label{svd}
\lim_{n\to\infty}\frac1{N^2}\sum_{i=1}^{N^2}F(\sigma_i(A^{(N)}))=\frac1{\mu_d(\Omega)}\int_\Omega F(|f(\bm x)|)\mathrm{d}\bm x,\qquad\forall\,F\in C_c(\mathbb R).
	\end{equation}
	In this case, $f$ is called the singular value symbol of $\{A^{(N)}\}_N$ and we write $\{A^{(N)}\}_N\sim_\sigma f$.
\end{itemize}
\end{definition}

	A sequence of matrices $\{Z^{(N)}\}_N$ such that $\{Z^{(N)}\}_N\sim_\sigma0$ is referred to as a zero-distributed sequence. In other words, $\{Z^{(N)}\}_N$ is zero-distributed iff
 \begin{equation}%
     \label{eq:zero_distributed_sequences_definition}%
\lim_{n\to\infty}\frac1{N^2}\sum_{i=1}^{N^2}F(\sigma_i(Z^{(N)}))=F(0),\qquad\forall\,F\in C_c(\mathbb R).
 \end{equation}
	Given a sequence of matrices $\{Z^{(N)}\}_N$, with $Z^{(N)}$ of size $N^2\times N^2$, the following property holds. In what follows, we use the natural convention $C/\infty=0$ for all numbers $C$.
	\begin{enumerate}[leftmargin=23pt]
		\item[\textbf{Z\,2.}]\label{Z2} $\{Z^{(N)}\}_N\sim_\sigma0$ if there exists a $p\in[1,\infty]$ such that $$\lim_{n\to\infty}N^{-2/p}\|Z^{(N)}\|_p=0,$$
		where $\|\cdot\|_p$ is the $p$-Schatten norm.
	\end{enumerate}

 \subsubsection{Approximating Classes of Sequences and Spectral Clustering}
	The space of matrix-sequences also presents  a metric structure, induced by a distance inspired from the concept of \textit{Approximating Class of Sequences} (a.c.s.). In fact, a sequence of matrix-sequences $\{B_m^{(N)}\}_N$ is said to be an a.c.s. for $\ms A$ if there exist $\{E_m^{(N)}\}_N$ and $\{R_m^{(N)}\}_N$ such that for every $m$ there exists $N_m$ with
	\[
	A^{(N)} = B_m^{(N)}  + E_m^{(N)} + R_m^{(N)}, \qquad \|E_m^{(N)}\|\le \omega(m), \qquad \rk(R_m^{(N)})\le N^2c(m)
	\]
	for every $N>N_m$, and
	\[
	\omega(m)\xrightarrow{m\to \infty} 0,\qquad c(m)\xrightarrow{m\to \infty} 0.
	\]
	In this case, we say that $\{B_m^{(N)}\}_N$ is a.c.s. convergent to $\ms A$, and we use the notation $\{B_m^{(N)}\}_N\acs \ms A$.
	In other words, $\{B_m^{(N)}\}_N$ converges to $\ms A$ if the difference $\{A^{(N)} -B_m^{(N)}\}_N$ can be decomposed into $\{E_m^{(N)}\}_N$ of 'small norm' and $\{R_m^{(N)}\}_N$ of 'small rank'.

We enunciate the main property of the a.c.s. we will need in the following.
	\begin{enumerate}[leftmargin=39pt]
		\item[\textbf{ACS\,4.}] \label{ACS4}Let $p\in[1,\infty]$ and assume for each $m$ there is $N_m$ such that, for $N\ge N_m$,
		$$
		\|A^{(N)}-B_m^{(N)}\|_p\le\varepsilon(m,N)N^{2/p}, \quad \lim_{m\to\infty}\limsup_{N\to\infty}\varepsilon(m,N)=0.
		$$
		Then $\{B_m^{(N)}\}_N\stackrel{\rm a.c.s.}{\longrightarrow}\{A^{(N)}\}_N$.
	\end{enumerate}
	
\begin{definition}[clustering of a sequence of matrices]\label{def-cluster}
Let $\{A^{(N)}\}_N$ be a sequence of matrices, with $A^{(N)}$ of size $N^2\times N^2$.
We say that $\{A^{(N)}\}_N$ is strongly clustered at zero (in the sense of the eigenvalues), or equivalently that the eigenvalues of $\{A^{(N)}\}_N$ are strongly clustered at zero, if, for every $\varepsilon>0$, the number of eigenvalues of $A^{(N)}$ of magnitude greater than $\varepsilon$ is bounded by a constant $C_\varepsilon$ independent of~$N$; that is, for every $\varepsilon>0$,
\begin{equation}\label{eq:stong_clustering_at_zero}
\#\{j\in\{1,\ldots,d_n\}:\ |\lambda_j(A_n)|>\varepsilon\}=O(1).
\end{equation}
By replacing ``eigenvalues'' with ``singular values'' and $\lambda_j(A_n)$ with $\sigma_j(A_n)$ in \eqref{eq:stong_clustering_at_zero}, we obtain the definitions of a sequence of matrices strongly clustered at zero in the sense of the singular values.
\end{definition}

	\subsubsection{Multilevel GLT}

	We now recall the theory of the multilevel generalized locally Toeplitz (GLT) sequences and symbols.
	A $2$-level GLT sequence $\{A^{(N)}\}_N$ is a special $2$-level matrix-sequence equipped with a measurable function $\kappa:[0,1]^2\times[-\pi,\pi]^2\to\mathbb C$, the so-called GLT symbol. Unless otherwise specified, the notation
 $$\{A^{(N)}\}_N\GLT\kappa$$
 means that $\{A^{(N)}\}_N$ is a $2$-level GLT sequence with symbol $\kappa$.
	We report here  the main properties of the GLT space we will make use of in this document.
	\begin{enumerate}
		\item[\textbf{GLT 2.}]\label{GLT2} If $\{A^{(N)}\}_N\sim_{\rm GLT}\kappa$ and $\{A^{(N)}\}_N=\{X^{(N)}\}_N+\{Y^{(N)}\}_N$, where
		\begin{itemize}
			\item every $X^{(N)}$ is Hermitian,
			\item $N^{-1}\|Y^{(N)}\|_F\to 0$,
		\end{itemize}
		then $\{A^{(N)}\}_N\sim_{\lambda}\kappa$.
		\item[\textbf{GLT 3.}]\label{GLT3} Here we list three important examples of GLT sequences.
		\begin{itemize}
			\item Given a function $f$ in $L^1([-\pi,\pi]^q)$, its associated Toeplitz sequence is $\{T^{(N)}(f)\}_N$, where the elements are multidimensional Fourier coefficients of $f$:
			\[ T^{(N)}( f ) = [ f_{\ii-\jj} ]^\nn_{\ii, \jj={\bf 1}}, \qquad f_\kk = \frac{1}{(2\pi)^q} \int_{-\pi}^{\pi} f(\btheta) e^{-\text i \kk\cdot\btheta} {{\rm d}}\theta. \]
			$\{T^{(N)}(f)\}_N$ is a GLT sequence with symbol $\kappa(\bfx,\btheta)=f(\btheta)$.
			\item Given an almost everywhere continuous function, $a:[0,1]^q\to\mathbb C$, its associated diagonal sampling sequence $\{D^{(N)}(a)\}_N$ is defined as
			\[ D^{(N)}(a) = \text{diag}\left(\left\{a\left(\frac \ii\nn\right) \right\}_{\ii=\bf 1}^\nn\right). \]
			$\{D^{(N)}(a)\}_N$  is a GLT sequence with symbol $\kappa(\bfx,\btheta)=a(\bfx)$.
			\item Any zero-distributed sequence $\{Z^{(N)}\}_N\svs0$ is a GLT sequence with symbol $\kappa(\bfx,\btheta)=0$ .
		\end{itemize}
		\item[\textbf{GLT 4.}]\label{GLT4} If $\{A^{(N)}\}_N\sim_{\rm GLT}\kappa$ and $\{B^{(N)}\}_N\sim_{\rm GLT}\xi$, then
		\begin{itemize}
			\item $\{(A^{(N)})^H\}_N\sim_{\rm GLT}\overline\kappa$, where $(A^{(N)})^H$ is the conjugate transpose of $A^{(N)}$,
			\item $\{\alpha A^{(N)}+\beta B^{(N)}\}_N\sim_{\rm GLT}\alpha\kappa+\beta\xi$ for all $\alpha,\beta\in\mathbb C$,
			\item $\{A^{(N)} B^{(N)}\}_N\sim_{\rm GLT}\kappa\xi$.
		\end{itemize}
		\item[\textbf{GLT 7.}]\label{GLT7} $\{A^{(N)}\}_N\sim_{\rm GLT}\kappa$ if and only if there exist GLT sequences $\{B_m^{(N)}\}_N\sim_{\rm GLT}\kappa_m$ such that $\kappa_m$ converges to $\kappa$ in measure and $\{B_m^{(N)}\}_N\xrightarrow{\text{a.c.s.}}\{A^{(N)}\}_N$ as $m\to\infty$.
	\end{enumerate}

\subsection{GLT symbol of Spherical IF and Convergence}

Each point of $S^2$ can be expressed in its spherical coordinates as
\[
r = (x_r,y_r,z_r) = (\theta_r,\vf_r), 	\qquad  x_r = \cos  \theta_r\cos  \vf_r, \quad  y_r = \sin  \theta_r\cos  \vf_r, \quad  z_r = \sin  \vf_r,
\]
where $\theta\in [0,2\pi)$ and $\vf\in [-\pi/2,\pi/2]$. In this case, the differential of the normalized area measure is
\begin{equation}
    \label{eq:solid_angle_measure_polar}d\sigma =\cos(\vf) d\theta d\vf/(4\pi).
\end{equation}
We can thus define a regular grid on $[0,2\pi)\times [-\pi/2,\pi/2]$ with parameter $N$ as
\begin{equation}
    \label{eq:center_points_grid_in_polar}
    z_{i,j} =\left(  (2i-1)h , -\pi/2 + (j-1/2)h \right), \quad i,j=1,2,\dots,N
, \qquad
h = \frac \pi  N \ll 1
\end{equation}
where the $z_{i,j}$ can be seen as the center of the rectangles
\begin{equation}
    \label{eq:rectangles_in_polar}
    S_{i,j} = [2(i-1)h,2ih] \times [-\pi/2 + (j-1)h ,-\pi/2 + jh].
\end{equation}

Moreover, let $d(r,w):= dist(r,w)= \arccos(\ip{r,w})$ be the arc-length distance (called also great-circle, orthodomic or spherical distance). It can be expressed in spherical coordinates as
and can be expressed as
\begin{equation}
    \label{eq:arc-length_in_polar}d(r,w) = \arccos(\sin \vf_r\sin \vf_w + \cos\vf_r\cos\vf_w\cos(\theta_r-\theta_w)).
\end{equation}

\subsubsection{Area and Diameter of $S_{i,j}$}

From classical Taylor expansion of trigonometrical functions, we get 
\begin{align}
  \label{eq:sin_taylor}  \sin(a+\ve) &=  \sin(a) + \ve\cos(a)- \frac {\ve^2}2 \sin(a)  - \frac {\ve^3}6 \cos(a) + \frac {\ve^4}{24} \sin(a) +
 O(\ve^5),\\
\label{eq:cos_taylor}  \cos(a+\ve) &=  \cos(a) - \ve\sin(a)- \frac {\ve^2}2 \cos(a) + O(\ve^3),\\
\label{eq:arccos_taylor}  \arccos(1 - \ve) &= \sqrt{2\ve} + O(\sqrt{\ve^3}), \qquad \sqrt {a+\varepsilon} = \sqrt a + \frac 1{2\sqrt a} O(\varepsilon).
\end{align}
Moreover, all the $O(	\ve^\alpha)$ terms above are actually bounded in absolute value by $c\ve^\alpha$ where $c$ is an absolute constant not depending on $\ve,\alpha$ or $a$. From now on the same will hold true for all the $O(\cdot)$ terms in this document.

\paragraph{Area:}

Let us now compute the area and diameter on $S^2$ for each $S_{i,j}$. Using the polar form of the uniform measure on $S^2$ as in \eqref{eq:solid_angle_measure_polar} we obtain
\begin{align*}
\sigma(S_{i,j})
 &= \int \chi_{S_{i,j}}(r) d\sigma(r)
= \frac 1{4\pi}\int_{-\pi/2}^{\pi/2} \int_0^{2\pi} \chi_{S_{i,j}}(r) \cos(\vf_r) d\theta_r d\vf_r
\end{align*}
and by the definition of $S_{i,j}$ in \eqref{eq:rectangles_in_polar},
\begin{align*}
\sigma(S_{i,j})
&= \frac 1{4\pi}
\int_{-\pi/2 + (j-1)h}^{-\pi/2 + jh} \int_{2(i-1)h}^{2ih}  \cos(\vf_r) d\theta_r d\vf_r\\
&=\frac h{2\pi}
\left( \cos( (j-1)h)-  \cos( jh)  \right).
\end{align*}
Eventually, exploiting the Taylor expansion of the cosine function \eqref{eq:cos_taylor},
\begin{equation}
    \label{eq:area_rectangles}\sigma(S_{i,j})
 = \frac {h^2}{2\pi}\sin(jh) + O(h^3).
\end{equation}

\paragraph{Diameter:}

The couple of most distant points within $S_{i,j}$ is $(v_{i,j},u_{i,j})$, the opposite corners of the rectangle, so its diameter is the arccosine of $d_{i,j} = \ip{v_{i,j},u_{i,j}}$, where, in polar coordinates,
\[
v_{i,j} = (2(i-1)h,-\pi/2 + (j-1)h), \quad
u_{i,j} = (2ih,-\pi/2 + jh)
\]
and by \eqref{eq:arc-length_in_polar},
\begin{align*}
d_{i,j} &= \cos ((j-1)h)\cos (jh) + \sin((j-1)h)\sin(jh)\cos(2h).
\end{align*}
Expanding both sine and cosine as in \eqref{eq:sin_taylor} and \eqref{eq:cos_taylor} respectively, we get
\begin{align*}
d_{i,j}
&= \Big[\cos(jh) +h\sin(jh)- \frac {h^2}2 \cos(jh)\Big]\cos (jh) \\
& \quad + \Big[\sin(jh)  - h\cos(jh)- \frac {h^2}2 \sin(jh)\Big]\sin(jh) \Big[1 - {2h^2}  \Big]
+ O(h^3)\\
&= 1 - \frac {h^2}2    -
{2h^2}  \sin(jh)^2
+ O(h^3)
\end{align*}
and using \eqref{eq:arccos_taylor}, we get the following formula for the diameter
\begin{align}\label{eq:diameter_of_rectangle}
\nonumber\textnormal{diam}(S_{i,j}) &=  \arccos (d_{i,j}) = \arccos \left(1 - \frac {h^2}2 - {2h^2}  \sin(jh)^2+ O(h^3)\right)\\
&= h\sqrt{ 1 + {4}  \sin(jh)^2 } +  O(h^2).
\end{align}

\subsubsection{Fixed Diagonal}

We now fix integers $(t,s)$ with $0\le |t|,|s|\le N$ to estimate the distance between $z_{i,j}$ and $z_{i-t,j-s}$ for $i,j,i-t,j-s$ positive integers less than $N$. We will only compute the first nonzero order of the distance with respect to $h$ in the limit $N\to \infty$ or equivalently $h\to 0$. First of all, it is easy to see that $d_{j,t,s}:=\ip{z_{i,j},z_{i-t,j-s}}$ does not depend on $i$. From \eqref{eq:center_points_grid_in_polar}, we get
\[
z_{i,j} = \left(  (2i-1)h , -\pi/2 + (j-1/2)h \right),
\quad
z_{i-t,j-s} =
\left(  (2i-2t-1)h , -\pi/2 + (j-s-1/2)h \right)
\]
so we can write down the dot product of the two points following \eqref{eq:arc-length_in_polar} as
\begin{align*}
d_{j,t,s} &= \cos ( (j-1/2)h )\cos ((j-s-1/2)h) + \sin( (j-1/2)h )\sin((j-s-1/2)h)\cos(2th).
\end{align*}
Using classical trigonometric identities, we can rewrite it as
\begin{align}\label{eq:distance_interm1}
\nonumber d_{j,t,s} &= \cos ( (j-1/2)h )\cos ((j-s-1/2)h) + \sin( (j-1/2)h )\sin((j-s-1/2)h)\cos(2th) \\
\nonumber&= \frac 12 \Big[  \cos(sh) + \cos ( (2j-s-1)h )  \Big]
 +\frac 12 \Big[  \cos(sh) - \cos ( (2j-s-1)h )  \Big]\cos(2th) \\
 &=  \cos(sh)
 -\frac 12 \Big[  \cos(sh) - \cos ( (2j-s-1)h )  \Big]\Big[ 1- \cos(2th)\Big].
\end{align}
Expanding the trigonometric functions following  \eqref{eq:sin_taylor}, \eqref{eq:cos_taylor}, we see that
\begin{align}\label{eq:distance_interm}
\nonumber d_{j,t,s} &=  \cos(sh)
 -\frac 12 \Big[  \cos(sh) - \cos ( (2j-s-1)h )  \Big]\Big[ 1- \cos(2th)\Big]\\\nonumber
 &=  1 - \frac 12 s^2h^2
 -\frac 12 \Big[  1 -   \cos ( 2jh ) +\sin(2jh)(s+1)h   \Big]\Big[ 2t^2h^2\Big]\\\nonumber
&\quad + s^4O(h^4) + t^4O(h^4) + s^2t^2O(h^4) + (s+1)^2t^2O(h^4)\\\nonumber
 &=  1 - \frac 12 h^2 \Big[  s^2 + 4t^2\sin(jh)^2 +2t^2(s+1)\sin(2jh)h  \Big]\\
&\quad + s^2\Big[ s^2+t^2\Big] O(h^4) + t^4O(h^4) + (s+1)^2t^2O(h^4).
\end{align}
Here we need to distinguish the cases in which $s=0$ from the cases where $|s|>0$.

\paragraph{Case $s=0$:}

Consider first the case $s=0$ and recall that $1\le i,j\le N$. From \eqref{eq:distance_interm1},
\begin{align*}
d_{j,t,0} &= 1  -\frac 12 \Big[  1 - \cos ( (2j-1)h )  \Big]\Big[ 1- \cos(2th)\Big].
\end{align*}
Notice that if we reflect with respect to the equator, i.e. we substitute $j$ with $N+1-j$, the quantity does not change since
\[\cos ( (2j-1)h ) = \cos ( 2\pi + (1-2j)h ) = \cos ( (2(N+1-j)-1)h ),\]
so from now on we suppose  $j\le N/2$ that implies, by concavity of the $\sin$ function on $[0,\pi/2]$, $\sin(jh)\ge 2 jh/\pi \ge 2h/\pi$.  By  \eqref{eq:distance_interm},
\begin{align*}
d_{j,t,0} &=  1 - \frac 12 h^2 \Big[   4t^2\sin(jh)^2 +2t^2\sin(2jh)h  \Big] + t^4O(h^4).
\end{align*}
Using \eqref{eq:arccos_taylor},
\begin{align}\label{eq:distance_ts_with_s_zero}\nonumber
d(z_{i,j},z_{i-t,j}) &=  \arccos (d_{j,t,0}) = \arccos \left(
1 -2t^2h^2 \sin(jh)^2 + t^2h^3 \sin(2jh) +  t^4O(h^4)
\right)\\\nonumber
&= \sqrt{4t^2h^2 \sin(jh)^2 - 2t^2h^3 \sin(2jh)+  t^4O(h^4) }
+ t^3O(h^3)
\\\nonumber
&= 2|t|h\sin(jh)
+
\frac{t^2\sin(2jh)O(h^3)  +  t^4O(h^4) }{4th\sin(jh)}
+ t^3O(h^3)
\\
&= 2|t|h\sin(jh)
+
tO(h^2)
+
\frac{  t^3O(h^3)}{2h/\pi}
+
t^3O(h^3)
= 2|t|h\sin(jh)
+
t^3O(h^2).
\end{align}
If we now consider the indices $j> N/2$, we find that
\begin{align*}
\nonumber d(z_{i,j},z_{i-t,j}) &=
d(z_{i,N+1-j},z_{i-t,N+1- j})= 2|t|h\sin((N+1-j)h)+ t^3O(h^2)\\
 &= 2|t|h\sin((j-1)h)+ t^3O(h^2)= 2|t|h\sin(jh)+ t^3O(h^2)
\end{align*}
so \eqref{eq:distance_ts_with_s_zero} holds for any $j$. Notice moreover that if $t$ is zero, then the above formula correctly reports that the distance is zero.

\paragraph{Case $s\ne 0$:}

Suppose now that $s\ne 0$ and in particular $|s|\ge 1$.
From \eqref{eq:distance_interm}
\begin{align*}
d_{j,t,s}
&=  1 - \frac 12 h^2 \Big[  s^2 + 4t^2\sin(jh)^2 +2t^2(s+1)\sin(2jh)h  \Big]\\
&\quad + s^2\Big[ s^2+t^2\Big] O(h^4) + t^4O(h^4) + (s+1)^2t^2O(h^4)\\
&=  1 - \frac 12 h^2 \Big[  s^2 + 4t^2\sin(jh)^2\Big] +t^2(s+1)O(h^3)  \\
&\quad + s^2\Big[ s^2+t^2\Big] O(h^4) + t^4O(h^4) + (s+1)^2t^2O(h^4).
\end{align*}
Since $|s|,|t|\le N$ and $h=\pi/N$, then $tO(h)$, $sO(h)$  and $(s+1)O(h)$ are all $O(1)$, so
\[
t^2(s+1)O(h^3)  + s^2\Big[ s^2+t^2\Big] O(h^4) + t^4O(h^4) + (s+1)^2t^2O(h^4)
= (|s|+|t|+1)^3 O(h^3)
\]
and thus
\begin{align*}
d_{j,t,s}
&=  \frac{h^2}2 \left[ s^2
  + 4t^2   \sin(jh)^2\right]
+(|s|+|t|+1)^3 O(h^3)
\\&= (s^2+t^2)O(h^2) +(|s|+|t|+1)^3 O(h^3) = (|s|+|t|+1)^2 O(h^2).
\end{align*}
Using now that $s\ne 0$ we see that $\left[ s^2
  + 4t^2   \sin(jh)^2\right]^{-1}\le 1/|s|\le 1$, so
\begin{align}\label{eq:distance_ts_with_s_non_zero}
\nonumber d(z_{i,j},z_{i-t,j-s}) &=  \arccos \left(1 -
\frac{h^2}2 \left[ s^2
  + 4t^2   \sin(jh)^2\right]
+ (|s|+|t|+1)^3 O(h^3)\right)\\\nonumber
&= h\sqrt{  \left[ s^2
  + 4t^2   \sin(jh)^2\right]
+ (|s|+|t|+1)^3 O(h)} +
(|s|+|t|+1)^3 O(h^3) \\\nonumber
&= h\sqrt{ s^2
  + 4t^2   \sin(jh)^2}
  + \frac{(|s|+|t|+1)^3}{\sqrt{ s^2  + 4t^2   \sin(jh)^2}}  O(h^2)
+  (|s|+|t|+1)^3 O(h^3)\\
&= h\sqrt{ s^2
  + 4t^2   \sin(jh)^2}
  + (|s|+|t|+1)^2 O(h^2).
\end{align}
Notice that we can combine  \eqref{eq:distance_ts_with_s_non_zero} and \eqref{eq:distance_ts_with_s_zero} into one formula that holds for any $s,t,j$, i.e.
\begin{align}\label{eq:distance_ts}
 d(z_{i,j},z_{i-t,j-s}) = h\sqrt{ s^2
  + 4t^2   \sin(jh)^2}
  + (|s|+|t|+1)^3 O(h^2).
\end{align}

\subsubsection{Normalization and Support of the Filter}

Let the filters $f_r(w)$ be defined as
\begin{equation*}
   f_r(w) := \frac 1C(R-\arccos(\ip{r,w}))^+ =\frac 1C(R-d(r,w))^+
\end{equation*}
where $r, w$ are points on $S^2$ and $(a)^+ := \max(0,a)$. $C$ is a constant meant to normalize the function $f_z(w)$ with respect to the measure $\sigma$. $R>0$ is the radius of the function, and it's supposed to be less than $\pi$, so that the support of $f_r$ is not the whole $S^2$.

\paragraph{Norm of the Filter:}

Since the norm of $f_r$ does not depend on the point $r$, one can suppose that $r=(0,\pi/2)$ are its polar coordinates, so its norm will be
\begin{align*}
\|f_r\| = 1 &= \frac 1C\int  (R-d(r,w))^+ d\sigma(w)
= \frac 1{4\pi C}\int_{-\pi/2}^{\pi/2} \int_0^{2\pi}  (R-d(r,w))^+ \cos(\vf_w) d\theta_w d\vf_w\\
&= \frac 1{4\pi C}\int_{-\pi/2}^{\pi/2} \int_0^{2\pi} (R-\pi/2 + \vf_w )^+ \cos(\vf_w) d\theta_w d\vf_w= \frac {1}{2C}\int_{\pi/2-R}^{\pi/2} (R-\pi/2 + \vf_w  )\cos(\vf_w)  d\vf_w\\
&= \frac {1}{2C}
\left[
[ (R-\pi/2 + \vf_w  )\sin(\vf_w)     ]_{\pi/2-R}^{\pi/2}
-
\int_{\pi/2-R}^{\pi/2} \sin(\vf_w)  d\vf_w
\right] = \frac 1{2C}
\left[ R  - \sin(R)   \right].
\end{align*}
This proves that $C = (R-\sin(R))/2$ and thus the final expression of the filter function is
\begin{equation}
    \label{eq:filter_definition}f_r(w)  =2\frac {(R-d(r,w))^+}{R-\sin(R)}.
\end{equation}

\paragraph{Support of the Filter:}

Notice that $f_r(w)\ne 0$ if and only if $d(r,w) < R$. Since $\arccos(x)$ is a strictly decreasing function on $[-1,1]$, then it is equivalent to
$$
\sin \vf_r\sin \vf_w + \cos\vf_r\cos\vf_w\cos(\theta_r-\theta_w) > \cos(R)\ge 1-R^2/2.
$$
As a consequence $f_r(w)\ne 0$ implies that
\begin{align}\label{eq:condition_for_points_in_support}\nonumber
\frac{R^2}2&> 1-\sin \vf_r\sin \vf_w - \cos\vf_r\cos\vf_w\cos(\theta_r-\theta_w) \\
\nonumber
&= 1- \cos(\Delta_\vf) + \cos\vf_r\cos\vf_w[1-\cos(\Delta_\theta)] \\
&\ge \frac 15 [\Delta_\vf^2 + \cos\vf_r\cos\vf_w\Delta_\theta^2]
\end{align}
where $\Delta_\vf = |\vf_r-\vf_w|$ and $\Delta_\theta = \min\{|\theta_r-\theta_w|,2\pi - |\theta_r-\theta_w|\}$ are both in $[0,\pi]$ and $1-\cos(x) \ge x^2/5$ for $x\in[0,\pi]$.
Suppose now that $R=mh$ with $m$ being an absolute constant. We take $r=z_{i,j}$ and $w=z_{i-t,j-s}$ with $i,j,i-t,j-s$ integers between $1$ and $N$, that in particular imply $|t|,|s|\le N-1$.
Call $\wt j := j-1/2\ge 1/2$ and assume that $f_r(w)\ne 0$.  From \eqref{eq:condition_for_points_in_support}, and \eqref{eq:center_points_grid_in_polar},
\begin{align*}
5R^2 = 5m^2h^2 &> 2 [\Delta_\vf^2 + \cos\vf_r\cos\vf_w\Delta_\theta^2]\\
&> 2 [s^2h^2 + \sin(\tilde jh) \sin((\tilde j-s)h)\min\{2|t|h,2\pi-2|t|h\}^2]\\
\implies 5m^2
&> 2 [s^2 + 4\sin(\tilde jh) \sin((\tilde j-s)h)\min\{|t|,N-|t|\}^2]\\
\implies 3m^2
&>  s^2 + 4\sin(\tilde jh) \sin((\tilde j-s)h)\min\{|t|,N-|t|\}^2
\end{align*}
Notice that in this case $|s|<\sqrt 3m$, i.e. $f_{z_{i,j}}(z_{i-t,j-s})$ is non-zero at most for a finite number of different values of the integer $s$. Since $h=\pi/N\le \pi$,
\begin{align*}
& \frac {\sin(x - h/2)}{\sin(x)} \ge \cos(h/2) - \sin(h/2) \tan(h)^{-1}= \frac 1{2\cos(h/2)} \ge \frac 12
 & \forall\,\, \pi \ge x\ge h
\end{align*}
so we find that
\begin{equation}\label{eq:nonzero_elem}
3m^2 >  s^2 + \sin(  jh )\sin(  (j-s)h )  \min\{|t|,N-|t|\}^2.
\end{equation}

\subsubsection{The Moving average and its Symbol Candidate}

Let us now consider the matrix $B$ as defined in \eqref{eq:sifting_matrix}, and substitute the rectangles $S_{i,j}$ defined in \eqref{eq:rectangles_in_polar}, whose centers are the points $z_{i,j}$, to obtain
\begin{equation}\label{eq:B_matrix}
    B_{(i,j),(p,q)}= \frac{1}{\sigma(S_{i,j})} \int_{S_{i,j}}\int_{S_{p,q}}f_w(r)d\sigma(r)d\sigma(w).
\end{equation}
The filter in \eqref{eq:filter_definition}  is continuous and symmetric in $w,r$, so for $h\to 0$ one has that $f_w(r)\approx f_{z_{p,q}}(z_{i,j})$ when $r\in S_{i,j}$ and $w\in S_{p,q}$, therefore
\[
B_{(i,j),(p,q)}\approx \sigma(S_{p,q})f_{z_{i,j}}(z_{p,q}).
\]
We can thus study the $N^2\times N^2$ matrix called $Op^{(N)}$ whose $(i,j),(p,q)$ element is
\begin{equation}
    \label{eq:sifting_matrix_op}Op^{(N)}_{(i,j),(p,q)} = \sigma(S_{p,q}) f_{z_{i,j}}(z_{p,q}) = 2\sigma(S_{p,q})\frac {(R-d(r,w))^+}{R-\sin(R)}.
\end{equation}
We will often omit the superscript ${(N)}$ in the following.

\begin{remark}
One could approximate the moving average of the signal more precisely as
\[
 \int_{S^2} s(w) f_z(w)d\sigma(w) \sim \sum_{p,q}  s(z_{p,q}) \int_{S_{p,q}} f_z(w)d\sigma(w),
\]
but it is immediate to see that
\[
\left| \sigma(S_{p,q}) f_z(z_{p,q}) -\int_{S_{p,q}} f_z(w)d\sigma(w)\right| \le \sigma(S_{p,q})\,\textnormal{diam}(S_{p,q}) = O(h^3)
\]
and it is zero whenever $\textnormal{dist}(z,S_{p,q})\ge R$.
\end{remark}

First we show that the matrix sequence $\{Op^{(N)}\}_N$ is zero-distributed as in \eqref{eq:zero_distributed_sequences_definition} and all its eigenvalues are strongly clustered at zero (see \eqref{eq:stong_clustering_at_zero}).

\begin{lemma}
If the radius $R$ of the filter does not depend on $N$, then the sequence $\{Op^{(N)}\}_N$ is zero-distributed and shows a strong cluster at zero both for the singular values and the eigenvalues. In particular its spectral symbol is the zero function.
\end{lemma}
\begin{proof}
Notice that from the definition of the filter \eqref{eq:filter_definition}, $|f_z(w)|\le R/C$ for any $z,w$, so the general element of the matrix is bounded by
\[
\left| Op^{(N)}_{(i,j),(p,q)}\right| \le \frac{1}{2\pi C}Rh^2\sin(qh) + O(h^3).
\]
As a consequence, $\|Op^{(N)}\|_F^2 = O(1)$ and this is sufficient to conclude that the sequence is zero distributed (by \hyperref[Z2]{\textbf{Z\,2}}) and shows a strong cluster at zero both for the singular values and the eigenvalues (See Theorems 2, 3 from \cite{al2014singular}). Its spectral symbol can be proved being the zero function by  \hyperref[GLT2]{\textbf{GLT 2}}.
\end{proof}

\noindent To get a better understanding of the spectral properties of the matrix, we have to suppose that $RN$ is an absolute constant for any $N$, meaning that $R$ depends on $N$.
 Once this assumption is made, we can fix a 2-level diagonal $(t,s)$ in $Op$, and focus on the set of entries $Op_{(i,j),(p,q)}$ for which $i-p=t$ and $j-q = s$. Our aim is to find a scalar function $a_{t,s}:[0,1]^2\to \f C$ such that $a_{t,s} (ih/\pi,jh/\pi)$ converges to $Op_{(i,j),(i-t,j-s)}$ uniformly in $i,j$.

\begin{lemma}
\label{lem:diagonal_functions_a_ts}  Suppose $R=mh$ with $m$ constant.  For any integer $t,s$ let
\begin{equation}
    \label{eq:ats}a_{t,s}(x_1,x_2) := 6 \sin(\pi x_2)  \frac{(m-
 \sqrt{ s^2
  + 4t^2   \sin(\pi x_2)^2}
)^+}{m^3\pi}.
\end{equation}
Then for fixed $t,s$
\[
\max_{i,j : 1\le i,j,i-t,j-s\le N} \left|a_{t,s} \left(\frac{ih}\pi,\frac{jh}\pi\right) -  Op_{(i,j),(i-t,j-s)}^{(N)}\right| = (|t|+1)^3O(h)\to 0.
\]
\end{lemma}

\begin{proof}
    Let us thus expand each entry $Op_{(i,j),(i-t,j-s)}$ for $h\to 0$.
Substituting \eqref{eq:area_rectangles} into \eqref{eq:sifting_matrix_op}, we get
\begin{align*}
Op_{(i,j),(i-t,j-s)}  &=\sigma(S_{i-t,j-s}) f_{z_{i,j}}(z_{i-t,j-s})\\
&=  2\left[\frac{h^2}{2\pi}\sin((j-s)h) + O(h^3)\right] f_{z_{i,j}}(z_{i-t,j-s}).
\end{align*}
Recall that by \eqref{eq:nonzero_elem}, if $s\ge \sqrt 3 m$ then $f_{z_{i,j}}(z_{i-t,j-s})$, and consequentially $Op_{(i,j),(i-t,j-s)}$, is zero and in this case
\[
a_{t,s} \left(\frac{ih}\pi,\frac{jh}\pi\right) = 6 \sin(jh)  \frac{(m-
 \sqrt{ s^2
  + 4t^2   \sin(jh)^2}
)^+}{m^3\pi} = 0
\]
so we only have to check the case $|s|<\sqrt 3 m = O(1)$. By \eqref{eq:sin_taylor} and \eqref{eq:filter_definition} we thus find that
\[
Op_{(i,j),(i-t,j-s)}
= 2\left[\frac{h^2}{2\pi}\sin(jh) + O(h^3)\right] \frac{(R- d(z_{i,j},z_{i-t,j-s})  )^+}{R-\sin(R)}
\]
and substituting $R=mh$,
\begin{align*}
Op_{(i,j),(i-t,j-s)}
&=  2\left[\frac{h^2}{2\pi}\sin(jh) + O(h^3)\right] \frac{(mh- d(z_{i,j},z_{i-t,j-s}) )^+}{m^3h^3/6 + O(h^5)}
\\
&=  12\left[\frac{1}{2\pi}\sin(jh) + O(h)\right] \frac{(m- d(z_{i,j},z_{i-t,j-s})/h )^+}{m^3 + O(h^2)}.
\end{align*}
Here we use $1/(m^3 +O(h^2)) = 1/m^3 + O(h^2)$, and $(m- d(z_{i,j},z_{i-t,j-s})/h )^+\le m=O(1)$ to see that
 \begin{align}
Op_{(i,j),(i-t,j-s)}
&=\frac{6}{\pi}\sin(jh) \frac{(m- d(z_{i,j},z_{i-t,j-s})/h )^+}{m^3}  + O(h)
.\label{eq:entries_of_Op}
\end{align}
Substituting now the equation for the distance \eqref{eq:distance_ts}, and using that $g(x):=(x)^+$ is a Lipschitz function, we get
 \begin{align*}
Op_{(i,j),(i-t,j-s)}
&=6\sin(jh) \frac{(m- \sqrt{s^2 +4t^2\sin(jh)^2} + (|s|+|t|+1)^3O(h)  )^+}{\pi m^3}  + O(h)
\\
&=6\sin(jh) \frac{(m- \sqrt{s^2 +4t^2\sin(jh)^2})^+ }{\pi m^3}  + (|s| + |t|+1)^3O(h)
\end{align*}
where $ (|s| + |t|+1)^3 =  ( |t|+1)^3O(1)$ since we are in the case $s=O(1)$.
After substituting $jh/\pi=x_2$, we eventually find the wanted function
\begin{equation*}
   a_{t,s}(x_1,x_2) := 6 \sin(\pi x_2)  \frac{(m-
 \sqrt{ s^2
  + 4t^2   \sin(\pi x_2)^2}
)^+}{m^3\pi}.
\end{equation*}
\end{proof}

With the functions $a_{t,s}(\bm x)$ as defined in \eqref{eq:ats}, let
\begin{equation}
    \label{eq:symbol}\kappa(\bm x,\bm\theta) := \sum_{t,s\in \mathbb Z} a_{t,s}(\bm x) \exp(\textnormal i(t\theta_1 + s\theta_2))
\end{equation}
where the infinite sum converges punctually everywhere for $\bm x\in [0,1]^2$,
since for every $\bm x$ there are at most a finite number of non-zero $a_{t,s}(\bm x)$, and thus a finite number of addends.
In fact $a_{t,s}(\bm x)\ne 0$ implies $\sin(\pi x_2)\ne 0$, and consequently
\begin{equation}
    \label{eq:infinite_sum_ats_is_finite}m> \sqrt{ s^2
  + 4t^2   \sin(\pi x_2)^2}
  \implies
  \frac{m^2 }{ 4  \sin(\pi x_2)^2}>
  t^2, \quad m^2>s^2
\end{equation}
thus bounding both $t$ and $s$ to a finite number of integers.
In the next sections, we prove that $\kappa(\bm x,\bm \theta)$ in \eqref{eq:symbol} is actually the GLT and spectral symbol for the matrix-sequence $\{Op^{(N)}\}_N$ when $RN$ is constant.

\subsubsection{GLT Symbol}

\begin{theorem}
    \label{thm:AN_has_right_symbol} If $RN$ is constant for any $N$, and $Op$ is as in \eqref{eq:sifting_matrix_op}, then
\[
    \{Op^{(N)}\}_N \sim_{GLT}\kappa(\bm x,\bm \theta) =  \sum_{t,s\in \mathbb Z} a_{t,s}(\bm x) \exp(\textnormal i(t\theta_1 + s\theta_2))
    \]
\end{theorem}

 \begin{proof}
Call $A^{(N)}_M$ the banded matrices
\[
A_M^{(N)}:=
 \sum_{|t|,|s|\le M} D^{(N)}(a_{t,s}(\bm x)) T^{(N)}(e^{\textnormal i(t\theta_1 + s\theta_2)})
\]
that are GLT sequences with symbols $\{A^{(N)}_M\}_N\GLT \kappa_M(\bm x,\bm\theta):=
\sum_{|t|,|s|\le M} a_{t,s}(\bm x)e^{\textnormal i(t\theta_1 + s\theta_2)}$ due to \hyperref[GLT3]{\textbf{GLT\,3}} and \hyperref[GLT4]{\textbf{GLT\,4}}. In
\eqref{eq:infinite_sum_ats_is_finite} we have already proved that $\kappa_M(\bm x,\bm\theta)$ converge everywhere to $\kappa(\bm x,\bm\theta)$, so we only need to show that
$\{A^{(N)}_M\}_N\acs \{Op^{(N)}\}_N$
since by \hyperref[GLT7]{\textbf{GLT\,7}} it would automatically prove the thesis.

Recall that by \eqref{eq:ats}
\[a_{t,s}(x_1,x_2) = 6 \sin(\pi x_2)  \frac{(m-
 \sqrt{ s^2
  + 4t^2   \sin(\pi x_2)^2}
)^+}{m^3\pi}\]
and by Lemma \ref{lem:diagonal_functions_a_ts}, \eqref{eq:nonzero_elem} and \eqref{eq:entries_of_Op},
\begin{align}
   \label{eq:appr1} Op_{(i,j),(i-t,j-s)}^{(N)} &= 6\sin(jh) \frac{(m- \sqrt{s^2 +4t^2\sin(jh)^2})^+ }{\pi m^3}  + (|t|+1)^3O(h)\\
    \label{eq:appr2}&=6\sin(jh) \frac{(m- d(z_{i,j},z_{i-t,j-s})/h )^+}{\pi m^3}  + O(h)
\end{align}
holds for any $i,j,t,s$ such that $1\le i,i-t,j,j-s\le N$,  $|t|<N$,  $|s|<\sqrt 3 m$, otherwise it is zero.
Suppose from now on that $\max\{1,\sqrt 3 m\}<M \le N$, so that we can write down the entries of $A^{(N)}_M$ as
\begin{equation}
    \label{eq:entries_of_ANM}
    [A_M^{(N)}]_{(i,j),(i-t,j-s)} =\begin{cases}
    a_{t,s}(ih/\pi,jh/\pi) = 6 \sin(jh)  \frac{(m-
 \sqrt{ s^2
  + 4t^2   \sin(jh)^2}
)^+}{m^3\pi},& |s|,|t|\le M,\\
0, & \text{otherwise.}
\end{cases}
\end{equation}
We now estimate the difference between $A_M^{(N)}$ and $Op^{(N)}$ as
\begin{align*}
    \|Op^{(N)} - A_M^{(N)}\|_F^2 &=
    \sum_{|t|,|s|\le M} \sum_{ \substack{i,j:\\1\le i,i-t,j,j-s\le N} }
    \left|Op^{(N)}_{(i,j),(i-t,j-s)} - [A_M^{(N)}]_{(i,j),(i-t,j-s)}\right|^2
    \\
    +& \sum_{N>|t|>M}\sum_{|s|\le \sqrt 3 m} \sum_{ \substack{i,j:\\1\le i,i-t,j,j-s\le N} }
    \left|Op^{(N)}_{(i,j),(i-t,j-s)}\right|^2
\end{align*}
where the second multi-sum has no elements with $|s|>M>\sqrt 3 m$ because that would lead to $Op^{(N)}_{(i,j),(i-t,j-s)}=0$. The first term is easy to bound thanks to \eqref{eq:appr1} and \eqref{eq:entries_of_ANM}.
\begin{align*}
    \sum_{|t|,|s|\le M} \sum_{ \substack{i,j:\\1\le i,i-t,j,j-s\le N} }
    \left|Op^{(N)}_{(i,j),(i-t,j-s)} - [A_M^{(N)}]_{(i,j),(i-t,j-s)}\right|^2
   = \sum_{|t|,|s|\le M} \sum_{ \substack{i,j:\\1\le i,i-t,j,j-s\le N} }
   (|t|+1)^6O(h^2) = O(1).
\end{align*}
For the second term, we notice that the expression in \eqref{eq:appr2} does not depend on $i$, since $d(z_{i,j},z_{i-t,j-s})=d_{j,t,s}$ as in \eqref{eq:distance_interm1}, and that for a fixed $t$, there are at most $n-|t|$ indices $i$ such that $1\le i, i-t\le N$, so
\begin{align*}
    &\sum_{N>|t|>M}\sum_{|s|\le M} \sum_{ \substack{i,j:\\1\le i,i-t,j,j-s\le N} }
    \left|Op^{(N)}_{(i,j),(i-t,j-s)}\right|^2 \\
    &= \sum_{N>|t|>M}\sum_{|s|\le M} \sum_{ \substack{i,j:\\1\le i,i-t,j,j-s\le N} }
  \frac{ 36\sin(jh)^2}{\pi^2 m^6} \left[(m- d(z_{i,j},z_{i-t,j-s})/h )^+\right]^2  + O(h)\\
  &\le \sum_{N>|t|>M}\sum_{|s|\le M} \sum_{ \substack{j:\\1\le j,j-s\le N} }
  (N-|t|)\frac{ 36\sin(jh)^2}{\pi^2 m^4}  + O(1)
  = \sum_{N>|t|>M} \sum_{ \substack{1\le j\le N} }
  (N-|t|)\frac{ 36(2M+1)\sin(jh)^2}{\pi^2 m^4}  + O(1)
\end{align*}
Now we split the sum in $j$ depending whether $\min\{j,N-j\}\le 2\sqrt 3 m$ or not. If it holds, then
\[\sin(jh) = \sin((N-j)h) \le 2\sqrt 3 mh = O(h) \]
and
\[
\sum_{N>|t|>M} \sum_{ j:\min\{j,N-j\}\le 2\sqrt 3 m }
  (N-|t|)\frac{ 36(2M+1)\sin(jh)^2}{\pi^2 m^4}  + O(1)
  =
  \sum_{N>|t|>M} \sum_{ j:\min\{j,N-j\}\le 2\sqrt 3 m }
   O(1) =O(N)
\]
Otherwise, we have $2\sqrt 3 m<j<N-2\sqrt 3 m$ and in particular,  $N>4\sqrt 3 m$. Keeping in mind that   $|s|\le \sqrt 3 m$, then
\begin{align*}
    \frac{\sin((j-s)h)}{\sin(jh)}& = \cos(sh)-\sin(sh)\tan(jh)^{-1} \ge
    \begin{cases}
        \cos(sh)-\sin(sh)\tan(2\sqrt 3 mh)^{-1}&s\ge 0\\
        \cos(sh)-\sin(sh)\tan(-2\sqrt 3 mh)^{-1}& s<0
    \end{cases}\\
    &=    \frac{ \sin((2\sqrt 3m -|s|)h)}{\sin(2\sqrt 3 mh)}\ge
    \frac{ \sin(\sqrt 3m h)}{\sin(2\sqrt 3 mh)} =
    \frac{ 1}{2\cos(\sqrt 3 mh)} \ge \frac 12.
\end{align*}
By the relation \eqref{eq:nonzero_elem}, the entry  $Op^{(N)}_{(i,j),(i-t,j-s)}$ is zero when
\begin{equation}
\label{eq:relation_for_zero_elements_when_j_is_large}     \min\{|t|,N-|t|\}^2 \sin(  jh )\sin((j-s)h) \ge  \min\{|t|,N-|t|\}^2 \sin(  jh )^2/2 \ge 3m^2 \ge 3m^2 -s^2
\end{equation}
so we can restrict the index $t$ to $ \min\{|t|,N-|t|\}< \sqrt 6m/\sin(jh)$ in the sum and find that
\begin{align*}
  &\sum_{N>|t|>M} \sum_{2\sqrt 3 m<j<N-2\sqrt 3 m }
  (N-|t|)\frac{ 36(2M+1)\sin(jh)^2}{\pi^2 m^4}  + O(1) \\
  &=\sum_{2\sqrt 3 m<j<N-2\sqrt 3 m }
  \sum_{\substack{N>|t|>M\\\min\{|t|,N-|t|\}< \sqrt 6m/\sin(jh)}}
  (N-|t|)\sin(jh)^2O(1)  + O(1)\\
  &\le\sum_{2\sqrt 3 m<j<N-2\sqrt 3 m }
  2\left[ N \left(\frac {\sqrt 6 m}{\sin(jh)} - M +1\right)^+ + \frac {6m^2}{\sin(jh)^2} \right]\sin(jh)^2O(1)  +   \frac 1{\sin(jh)}O(1)\\
&\le\sum_{1\le j< N}
   \left(\sqrt 6 m - (M-1)\sin(jh)\right)^+ \sin(jh)O(N) + O(1)  +   \frac 1{\sin(jh)}O(1).
\end{align*}
Since $\sin(jh) = \sin((N-j)h)$, we can bound the sum by doubling it and stop  the index at $j=\lfloor N/2\rfloor$. Under this hypothesis we have the bounds $jh\ge \sin(jh)\ge 2jh/\pi$  leading to $\sum_{j=1}^{\lfloor N/2\rfloor} 1/\sin(jh)\le \sum_{j=1}^{\lfloor N/2\rfloor} \pi/(2jh) = O(N\log(N))$ and
\begin{align*}
  &\sum_{N>|t|>M} \sum_{2\sqrt 3 m<j<N-2\sqrt 3 m }
  (N-|t|)\frac{ 36(2M+1)\sin(jh)^2}{\pi^2 m^4}  + O(1) \\
&\le
\left[ \sum_{1\le j\le \lfloor N/2\rfloor}
   \left(\sqrt 6 m - \frac{2(M-1)h}\pi j\right)^+ jO(1) \right] +   O(N\log(N)).
   \\
&\le
\left[ \sum_{1\le j< \frac{\sqrt 6\pi m}{2(M-1)h}}
   \sqrt 6 m \cdot jO(1) \right]+   O(N\log(N))
  \\& \le \frac{ 6\pi^2 m^2}{4(M-1)^2h^2}  O(1) +   O(N\log(N))
   =\frac{1}{(M-1)^2}  O(N^2) +   O(N\log(N)).
\end{align*}
Putting all terms together, we find that
\[
 \|Op^{(N)} - A_M^{(N)}\|_F^2 = O(1) + O(N) + \frac{1}{(M-1)^2}  O(N^2) +   O(N\log(N)) = \left[ \frac{1}{(M-1)^2}  O(1) +   O\left(\frac{\log(N)}{N}\right)  \right] \cdot N^2
\]
and we conclude thanks to \hyperref[ACS4]{\textbf{ACS\,4}} with $p=2$ and $$\lim_{M\to\infty}\limsup_{N\to\infty}\varepsilon(M,N) = \lim_{M\to\infty}\limsup_{N\to\infty}\sqrt {\frac{1}{(M-1)^2}  O(1) +   O\left(\frac{\log(N)}{N}\right)} = 0.$$
\end{proof}

\subsubsection{Spectral Symbol}

Here we show that $\{Op^{(N)} \}_N$ has spectral symbol $\kappa(\bm x,\bm\theta)$. To do so, we need to prove that it is close enough to an Hermitian matrix in Frobenius norm, i.e. its skew-Hermitian part is small, and conclude thanks to \hyperref[GLT2]{\textbf{GLT\,2}}.

\begin{theorem}
  \label{th:spectral_symbol}     If $RN$ is constant for any $N$, and $Op$ is as in \eqref{eq:sifting_matrix_op}, then
\[
    \{Op^{(N)}\}_N \sim_{\lambda}\kappa(\bm x,\bm \theta) =  \sum_{t,s\in \mathbb Z} a_{t,s}(\bm x) \exp(\textnormal i(t\theta_1 + s\theta_2))
    \]
\end{theorem}

\begin{proof}
    Let $E^{(N)}$ be 2 times the skew-Hermitian part of $Op^{(N)}$, i.e. $E^{(N)}:= Op^{(N)} - [Op^{(N)}]^H$. By \eqref{eq:nonzero_elem}, if $|s|\ge \sqrt 3m$ then the elements  $Op^{(N)}_{(i,j),(i-t,j-s)}$ and $[Op^{(N)}]^H_{(i,j),(i-t,j-s)} = Op^{(N)}_{(i-t,j-s),(i,j)}$ are both zero, so $E^{(N)}_{(i,j),(i-t,j-s)}$ is zero too and thus we can focus on the case $s<\sqrt 3 m=O(1)$.
    By definition \eqref{eq:sifting_matrix_op}, equation \eqref{eq:entries_of_Op} and using that the distance $d(r,w)$ is symmetric in $r,w$, we have
\begin{align*}
  E^{(N)}_{(i,j),(i-t,j-s)} = 6[\sin(jh)-\sin((j-s)h)] \frac{(m-d(z_{i,j},z_{i-t,j-s})/h)^+}{\pi m^3} + O(h)
    \end{align*}
but since $|s|<\sqrt 3 m$,
\begin{align*}
  |\sin(jh)-\sin((j-s)h)| = |\sin(jh)[1-\cos(sh)]+\cos(jh)\sin(sh)|
  \le |1-\cos(sh)|+|\sin(sh)| \le \frac {s^2h^2}2 + sh
\end{align*}
and thus $ E^{(N)}_{(i,j),(i-t,j-s)}=O(h)$. As a consequence, its Frobenius norm is bounded by
\[
\|E^{(N)}\|_F^2 = \sum_{|s|<\sqrt 3 m} \sum_{|t|<N} \sum_{ \substack{i,j:\\1\le i,i-t,j,j-s\le N} }
    \left|E^{(N)}_{(i,j),(i-t,j-s)}\right|^2
    =
    N^3O(h^2) = O(N) = o(N^2).
\]
If $\wt{Op}^{(N)} := (Op^{(N)}+[Op^{(N)}]^H)/2 $ is the Hermitian part of $Op^{(N)}$, then
\[
\|Op^{(N)} - \wt{Op}^{(N)}\|_F^2 =
\frac 14\|E^{(N)}\|_F^2 = o(N^2)
\]
and since $ \{Op^{(N)}\}_N \sim_{GLT}\kappa(\bm x,\bm \theta)$ by Theorem \ref{thm:AN_has_right_symbol}, then $\kappa(\bm x,\bm \theta)$ is also the spectral symbol of $\{Op^{(N)}\}_N$ thanks to \hyperref[GLT2]{\textbf{GLT\,2}}.
\end{proof}

\subsubsection{Counterexample to Convergence}


Already for $m=2$ the symbol $\kappa(\bm x,\bm \theta)$ in Theorem \ref{th:spectral_symbol} is negative somewhere on its domain. From \eqref{eq:ats},
\[
\kappa(\bm x,\bm \theta) = \sum_{t,s\in \mathbb Z}
3 \sin(\pi x_2)  \frac{(2-
 \sqrt{ s^2
  + 4t^2   \sin(\pi x_2)^2}
)^+}{4\pi}
\exp(\textnormal i(t\theta_1 + s\theta_2))
\]
but since $m=2$, the only nonzero terms are for $s=1,0,-1$, i.e.
\[
\kappa(\bm x,\bm \theta) = \frac{3 \sin(\pi x_2)}{4\pi} \sum_{t\in \mathbb Z}
\exp(\textnormal it\theta_1)  \left[
 2\cos(\theta_2) (2-
 \sqrt{ 1
  + 4t^2   \sin(\pi x_2)^2}
)^+
+
 2 (1-
 |t|   \sin(\pi x_2)
)^+
\right]
\]
and for $\bm\theta = (0,\pi)$ we can further simplify into
\[
\kappa(\bm x,(0,\pi)) = \frac{3 \sin(\pi x_2)}{2\pi} \sum_{t\in \mathbb Z}
  (1-
 |t|   \sin(\pi x_2)
)^+
-   \left(2-
 \sqrt{ 1
  + 4t^2   \sin(\pi x_2)^2}
\right)^+.
\]
For $x_2\to 0$, both terms in the sum become quadrature formulas that converge to definite integrals. In particular,
\[
\sin(\pi x_2) \sum_{t\in\mathbb Z}   (1- |t|   \sin(\pi x_2) )^+
\xrightarrow[]{x_2\to 0}
\int_{-1}^1 1-|x|\, dx = 1
\] \[
\sin(\pi x_2) \sum_{t\in\mathbb Z}
\left(2-
 \sqrt{ 1
  + 4t^2   \sin(\pi x_2)^2}
\right)^+
\xrightarrow[]{x_2\to 0}
\int_{-1/2}^{1/2} 2 - \sqrt{1+4x^2}\, dx = 2 - \frac \pi 4
\]
and ultimately,
\[
\kappa(\bm x,(0,\pi)) \xrightarrow[]{x_2\to 0} \frac 3{2\pi} \left (\frac\pi 4 -1\right) < -0.1
\]
so $\kappa(\bm x,\bm \theta)$ is negative in an open neighbourhood of the line $((x_1,0),(0,\pi))\in [0,1]^2\times [-\pi,\pi]^2$.

\section{Convergent extension of Iterative Filtering on the sphere}

From all previous results, it is clear that for generic filters, like cone filters, the extension of Iterative Filtering to spherical geometry is not convergent in general. It is possible to create a parallelism between the extension of Iterative Filtering to spherical geometry and the so-called Adaptive Local Iterative Filtering (ALIF) for 1D data. In both cases, the filter changes size and shape point by point. In the case of spherical data this comes as a consequence of the discretization we choose for the sphere. In ALIF this is the main feature of the algorithm. If we vectorize the data defined on the sphere that we want to decompose and treat them as if they were a 1D signal, the decomposition approach can be reformulated as a special case of the ALIF algorithm.

The convergence of ALIF, at least in the formulation presented in \cite{cicone2016adaptive} where the scaling of the filter is done linearly point by point, is not known. But it can be made convergent by multiplying the transpose of the operator associated with the filter with the operator itself. This is what is done in the so-called SALIF algorithm \cite{barbarino2021stabilization}. The idea is that the iterative sifting operator $I-B_m$ applied to a discretized signal $g_m$  tends to extract the component of $g_m$ in a neighbourhood of the kernel of $B_m$ when the condition for convergence is satisfied, i.e. $\rho(I-B_m)\le 1$ and $\lambda \in \Lambda(I-B_m)$, $|\lambda|=1 \implies \lambda =1$. If we modify the sifting operator into $I - B_m^TB_m$ we notice that the convergence condition is now satisfied since $B_m$ is stochastic, and the kernel of $B_m$ is equal to the kernel of $B_m^TB_m$.

Following the same approach, in the extension of iterative filtering to spherical geometry we can substitute the operator $B$ with the operator $B^TB$. In doing so all eigenvalues become real, positive, and contained in the interval $[ 0,\, 1]$. Hence, the algorithm converges a priori. This is what we call from now on the Spherical Iterative Filtering (SIF) technique. The SIF pseudocode is reported in Algorithm \ref{alg:SIF}.

\begin{algorithm}
\caption{\textbf{Spherical Iterative Filtering} IMF = SIF$(g)$}\label{alg:SIF}
\begin{algorithmic}
\STATE IMF = $\left\{\right\}$
\WHILE{the number of extrema of $g$ $\geq 2$}
      \STATE $g_1 = g$
      \WHILE{the stopping criterion is not satisfied}
                  \STATE  compute the filter length $l_m$ for $g_{m}(x)$ and generate the corresponding convolution matrix $F_m$
                  \STATE  $g_{m+1}(x) = (I-B_m^TB_m) g_{m}(x)$
                  \STATE  $m = m+1$
      \ENDWHILE
      \STATE IMF = IMF$\,\cup\,  \{ g_{m}\}$
      \STATE $g=g-g_{m}$
\ENDWHILE
\STATE IMF = IMF$\,\cup\,  \{ g\}$
\end{algorithmic}
\end{algorithm}

\section{Numerical examples}\label{sec:Num_Ex}

In this section, we run numerical tests of the theoretical results presented in this work\footnote{Codes are available at \url{www.cicone.com}}. In particular, in a first test we study the spectrum of the matrix $B$ associated with the generalization of Iterative Filtering algorithm via straight isotropic convolution. In a second test we apply both the direct generalization of Iterative Filtering algorithm via isotropic convolution and the proposed convergent Spherical Iterative Filtering method to an artificial signal.

\subsection{Test 1}

In this first test, we evaluate the eigenvalues of the matrix $B$, as defined in \eqref{eq:B_matrix}, for the radius value of $\frac{\pi}{10}$ and a mesh grid as in \eqref{eq:center_points_grid_in_polar} with $N=100$.

\begin{figure}[ht]
    \centering
    \includegraphics[width=0.45\linewidth]{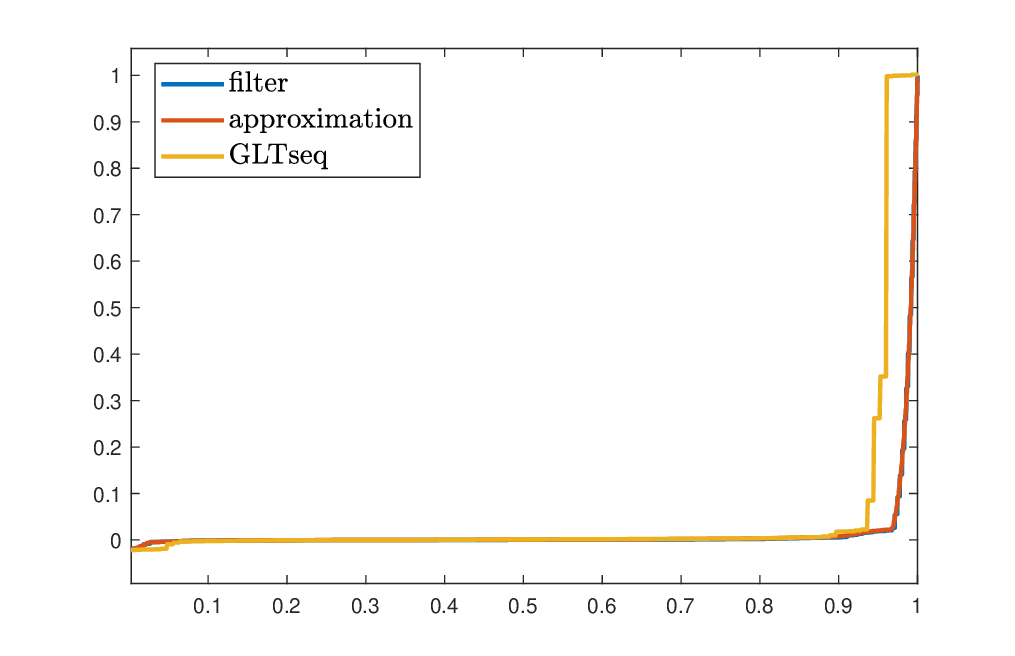}
    \includegraphics[width=0.45\linewidth]{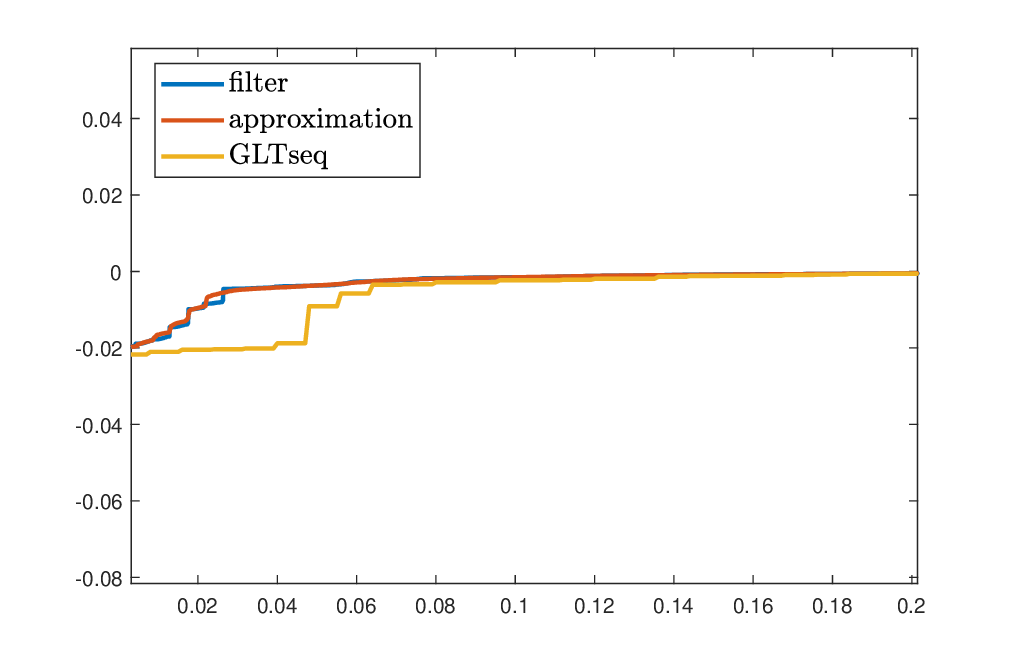}
    \caption{In the left panel, we find the real part of the eigenvalues of $B$ obtained using: double integration, solid blue; the approximation described in \eqref{eq:entries_of_Op}, solid red; and GLT sequence approximation, solid yellow. In the right panel we find a zoomed-in version of the same plot on the horizontal interval $[0,\, 0.2]$.}
    \label{fig:eig}
\end{figure}

In Figure \ref{fig:eig} we report the real value of the eigenvalues of $B$ in increasing order computed using double integration, reported in solid blue, the approximation described in \eqref{eq:entries_of_Op}, solid red, and the approximation obtained using the GLT sequence. This last one, in particular, is known to represent the spectrum of the matrix up to a $o(n)$ of eigenvalues which cannot be approximated. If we zoom in the horizontal axis on the interval $[0,\, 0.2]$, we can see that all approaches confirm the presence of eigenvalues with negative real parts. If we increase the grid size to $N=200$ and $N=1000$ we confirm the presence of eigenvalues with negative real parts, Figure \ref{fig:GLT}.

We recall that for the spectrum of $B^TB$, which is the stabilized version of the operator $B$ which is used in the SIF algorithm, there is no need  to approximate its eigenvalues since it is known a priori that the are all real and positive.

\begin{figure}[ht]
    \centering
    \includegraphics[width=0.45\linewidth]{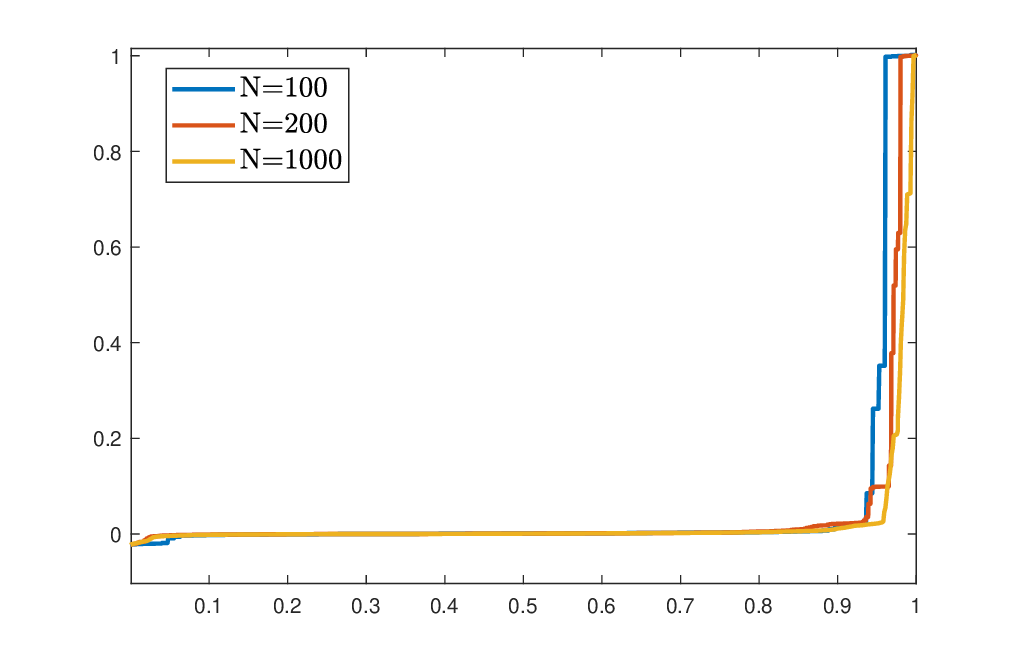}
    \includegraphics[width=0.45\linewidth]{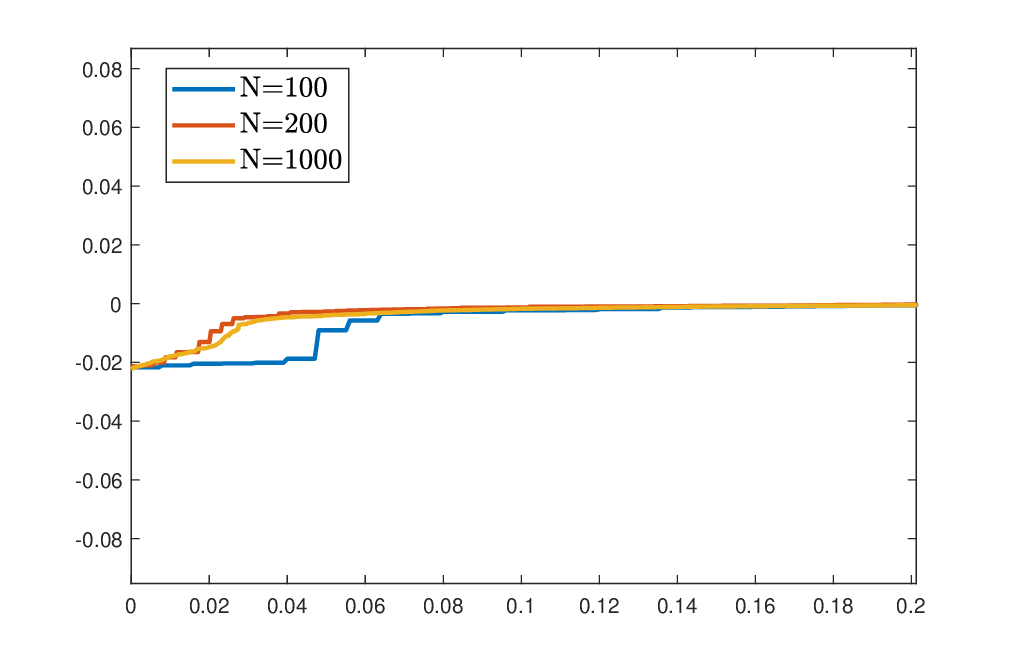}
    \caption{In the left panel we show the real part of the eigenvalues for the GLT sequences associated with mesh grids with increasing precision. In the right panel we find a zoomed-in version of the same plot on the horizontal interval $[0,\, 0.2]$.}
    \label{fig:GLT}
\end{figure}

\subsection{Test 2}

In this second test, we apply the generalization of Iterative Filtering algorithm via straight isotropic convolution and the SIF method, presented in Algorithm \ref{alg:SIF}, to an artificial signal defined on the sphere, ref. Figure \ref{fig:testNonConv} left panel. The artificial signal is constructed to contain two circular waves with different periodicities centered at different locations on the sphere.

\begin{figure}[ht]
    \centering
    \includegraphics[width=0.3\linewidth]{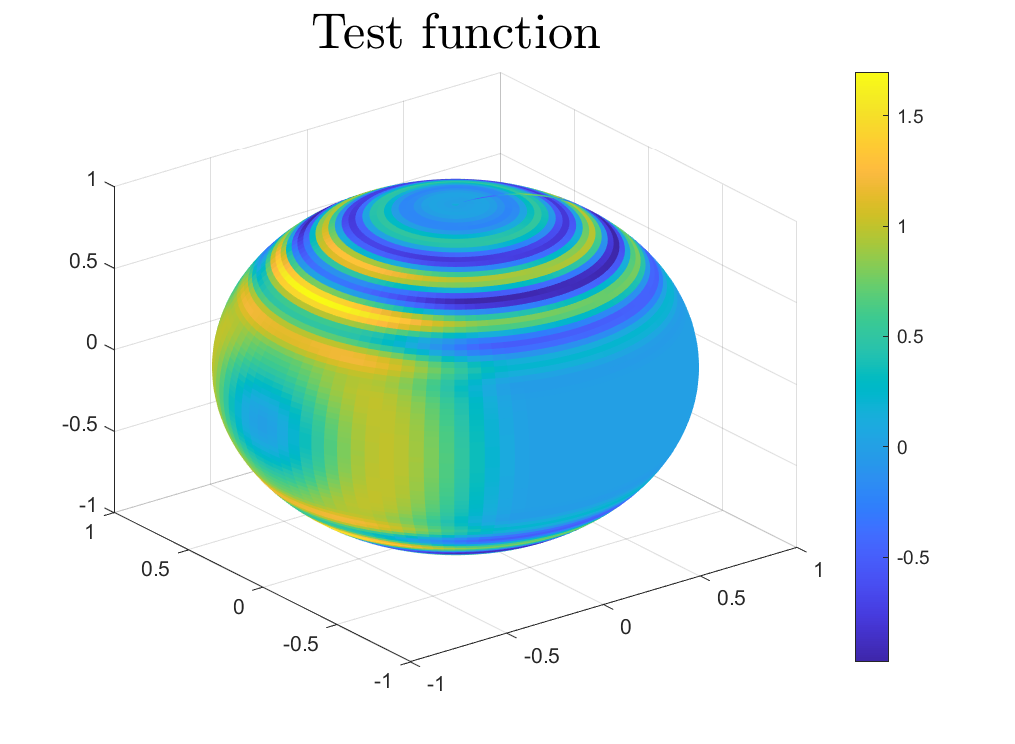}
    \includegraphics[width=0.3\linewidth]{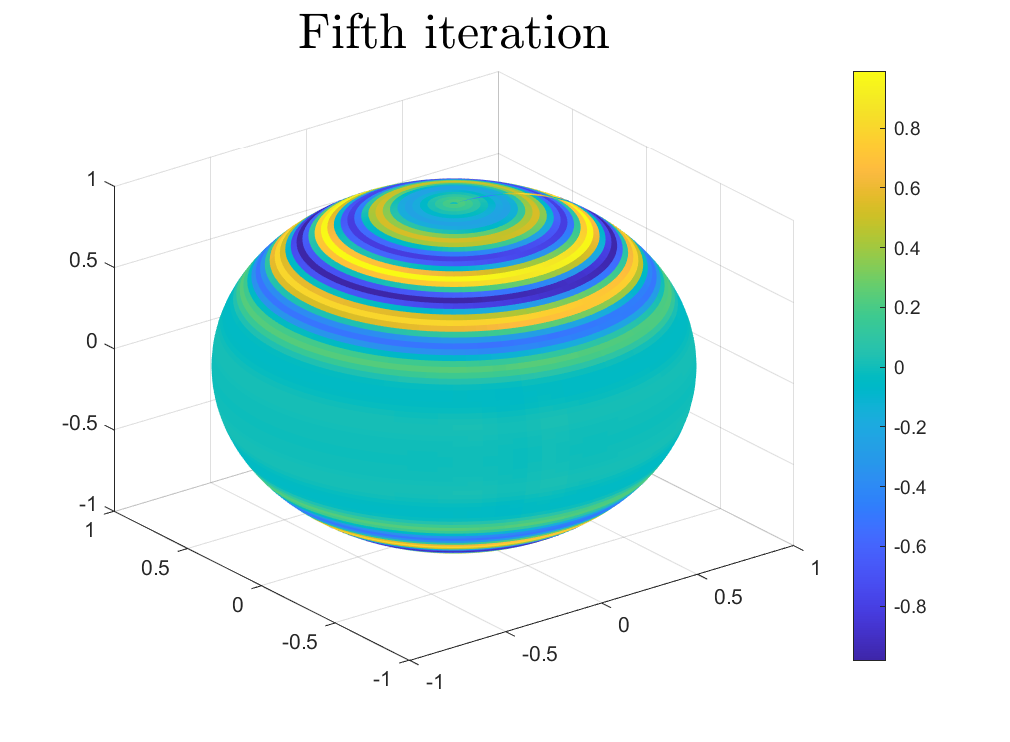}
    \includegraphics[width=0.3\linewidth]{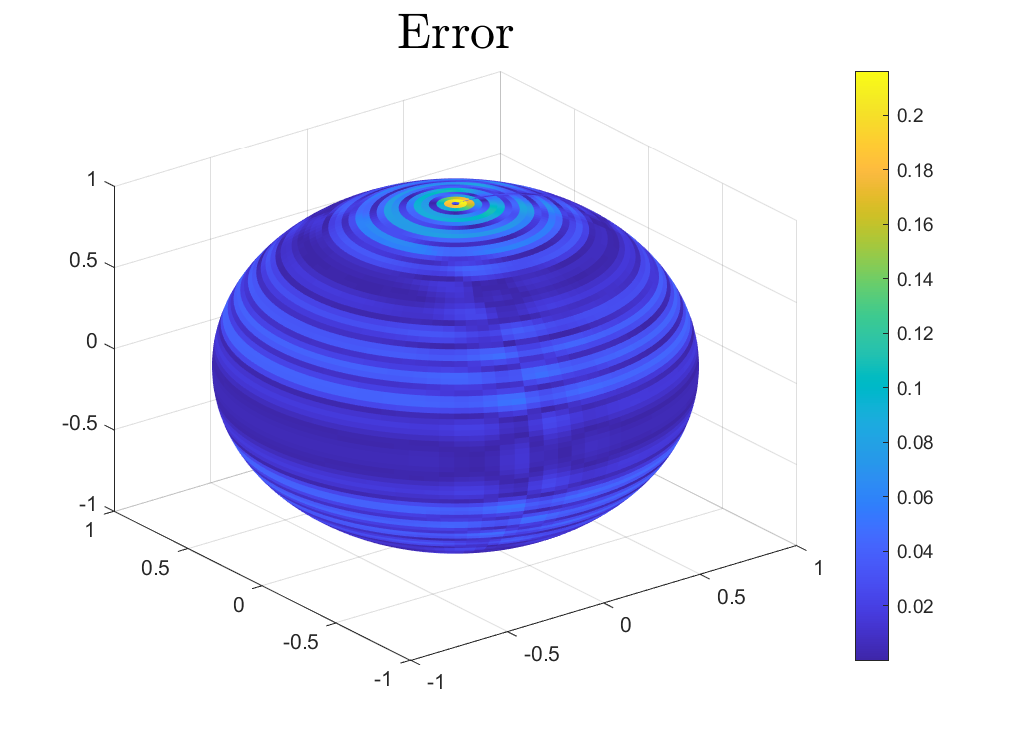}
    \caption{Left panel: test signal containing two circular waves of different periodicity. Central panel: first component obtained applying the Sifting operator $I-B$ to the signal shown on the left panel for $5$ iterations. Right panel: the error computed as the absolute value of the difference between the component obtained after $5$ iterations and the ground truth component we want to extract.}
    \label{fig:testNonConv}
\end{figure}

We first apply the nonconvergent matrix $B$ defined in \eqref{eq:B_matrix} for the radius value of $\frac{\pi}{20}$ and a mesh grid as in \eqref{eq:center_points_grid_in_polar} with $N=100$. The real parts of its spectrum elements, approximated using the approaches reviewed in Test 1, is shown in the right panel in Figure \ref{fig:itErr}. Results of the application of 5 iterations of the generalization of Iterative Filtering algorithm via straight isotropic convolution are shown in the center and right panel of Figure \ref{fig:testNonConv}. In particular, the central panel shows the first IMF extracted by the algorithm after 5 iterations, whereas the right panel represent the difference between this extracted IMF and the ground truth component. It is important to mention that the stopping condition \eqref{eq:stopping_criterion}, which is necessary in order to stop the iterative application of the Sifting operator, is not satisfied for any iteration up to 200 and the reported component in the central panel of Figure \ref{fig:testNonConv} is the closest sifted signal to the ground truth reached by the algorithm in all the 200 iterations.  This is another consequence of the divergence of the operator $B$.

\begin{figure}[ht]
    \centering
    \includegraphics[width=0.3\linewidth]{testStart.eps}
    \includegraphics[width=0.3\linewidth]{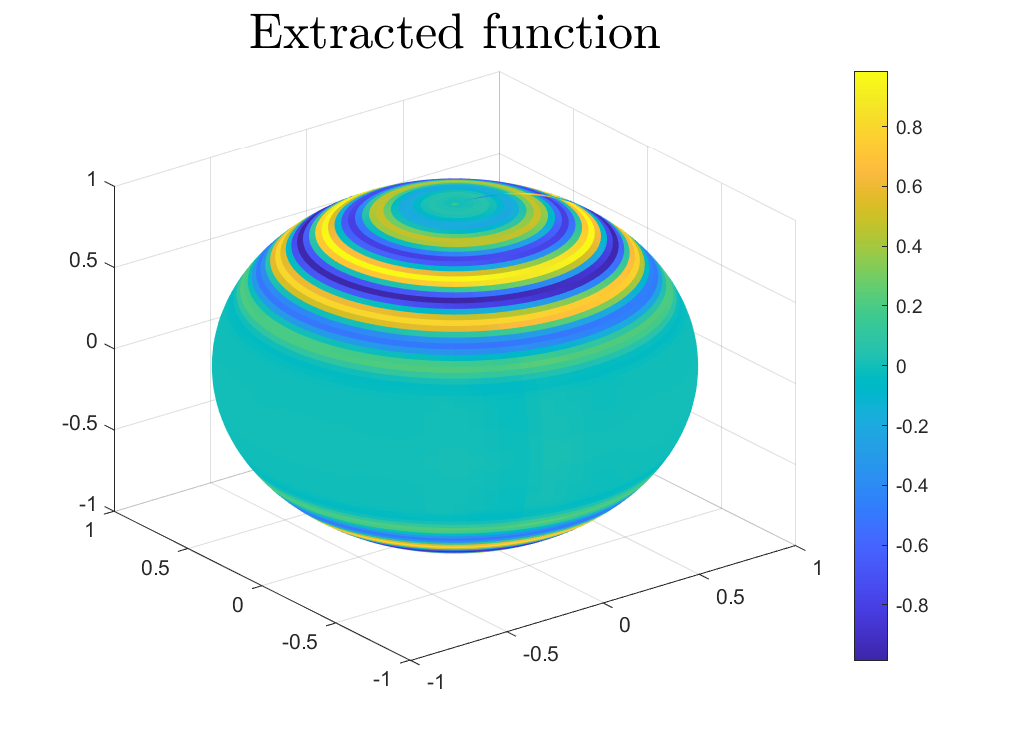}
    \includegraphics[width=0.3\linewidth]{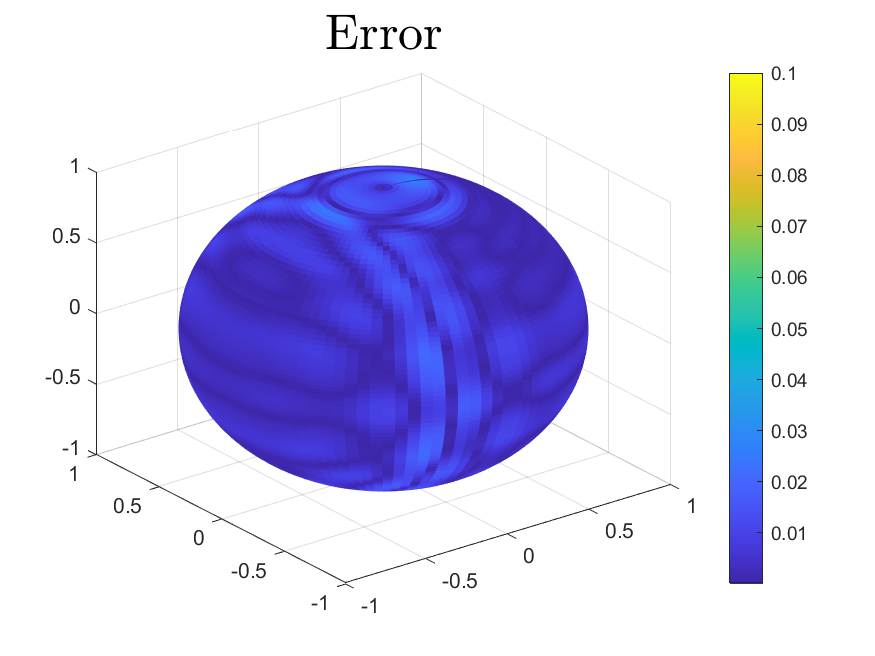}
    \caption{Left panel: test signal containing two circular waves of different periodicity. Central panel: first component obtained applying SIF to the signal shown on the left panel for $5$ iterations. Right panel: the error computed as the absolute value of the difference between the IMF obtained after $5$ iterations and the ground truth component.}
    \label{fig:testConv}
\end{figure}

In Figure \ref{fig:testConv} we report the results obtained by applying the convergent SIF method to the same signal until the stopping criterion \eqref{eq:stopping_criterion} is satisfied. In this case the algorithm iterates 5 times. From the right panel in Figure \ref{fig:testConv}, we can see that the error is drastically reduced. It is interesting to notice that the error is stronger in the region of the sphere where the two waves meet.

\begin{figure}[ht]
    \centering
    \includegraphics[width=0.4\linewidth]{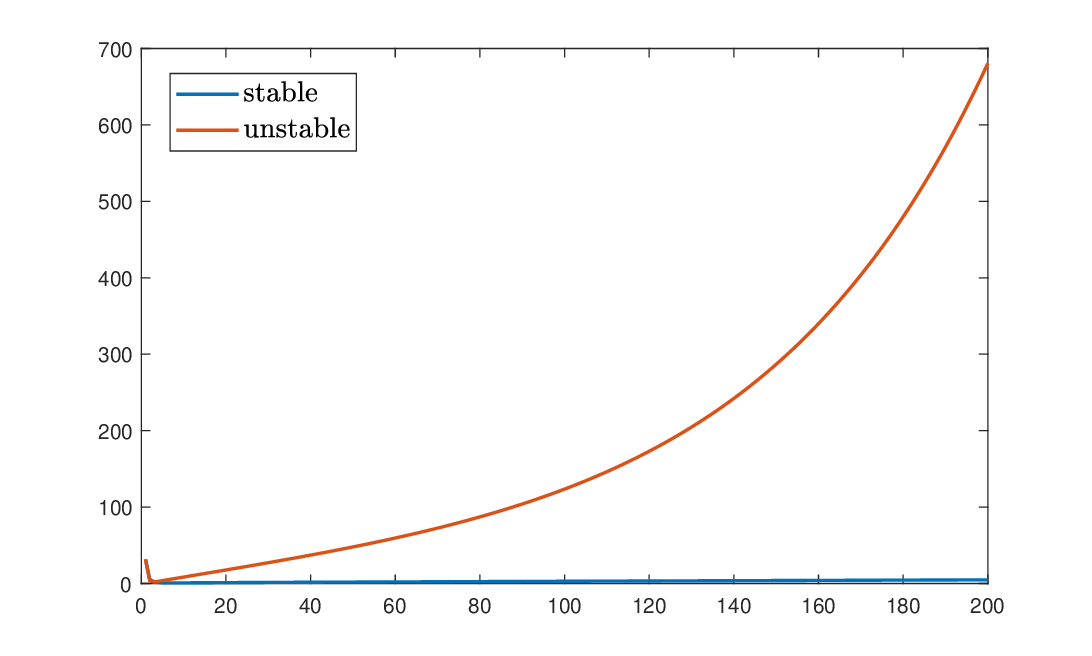}
    \includegraphics[width=0.4\linewidth]{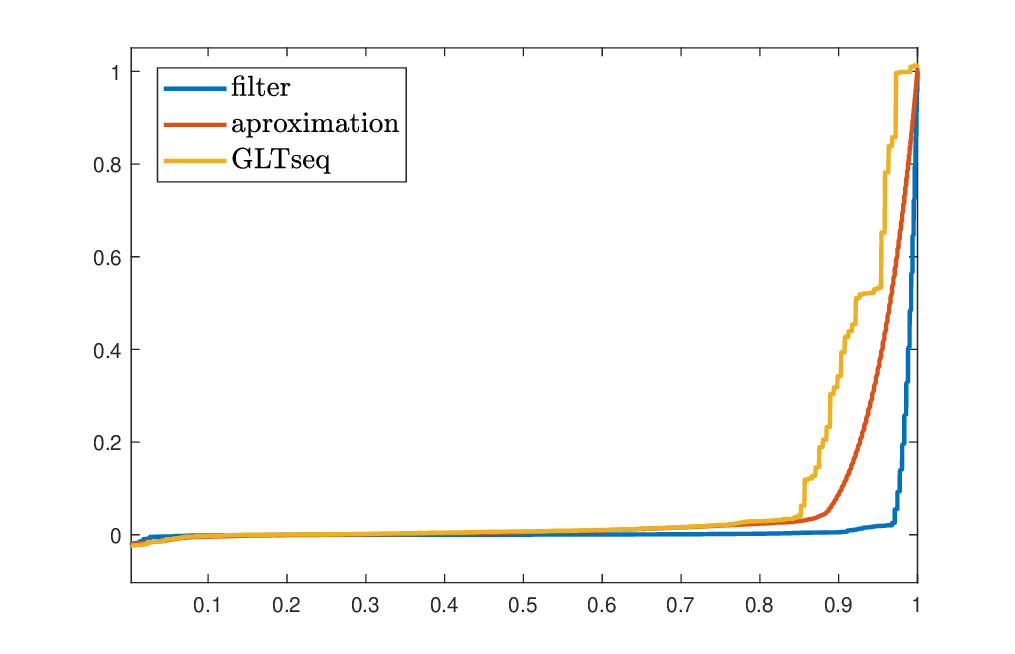}
    \caption{Left panel: $L^2$ norm curves of the error for the first $200$ iterations of the generalization of Iterative Filtering algorithm via straight isotropic convolution and SIF. Right panel: a plot of the real part of the eigenvalues for the matrix $B$ associated with a conic filter of radius $\frac{\pi}{20}$ as computed via double integration, solid blue, the approximation described in \eqref{eq:entries_of_Op}, solid red, and GLT sequence approximation, solid yellow.}
    \label{fig:itErr}
\end{figure}

If we let both algorithms to run for 200 iterations and measure the $L^2$ norm of the error between the computed IMF and its ground truth, we obtain the curves shown in the left panel of Figure \ref{fig:itErr}. It is evident from these curves that the generalization of Iterative Filtering algorithm via straight isotropic convolution is unstable and causes a strong energy injection, while the error associated with the SIF decomposition is bounded, having its minimum when the stopping criterion is achieved after 5 iterations. The right panel of Figure \ref{fig:itErr} confirms the presence of eigenvalues with negative real parts in the spectrum of the matrix $B$ used in this example.

\section{Conclusions}\label{sec:Conclusions}

Given the importance of developing new nonlinear algorithms for the decomposition of spherical data sets, in this work we tackled the problem of extending the Iterative Filtering (IF) algorithm to the case of the sphere and studying its a priori convergence. After reviewing the basic properties of IF for the one-dimensional case, we introduce its extension to spherical domains in the continuous and discrete case. We leverage on the Generalized Locally Toeplitz (GLT) theory to study the convergence property of this extension of IF to the sphere. In particular, after recalling properties of 2-level GLT sequences, we have studied how to characterize spectrally the matrix associated to the discrete sifting operator in the spherical iterative filtering. From this analysis we discover that the spherical iterative filtering algorithm is not guaranteed to converge, at least if we use conic filters. Following what was done in the literature for stabilizing the adaptive local iterative filtering, in this work we propose to stabilize the algorithm by multiplying the sifting operator times its transpose on the left. In doing so, the spectrum of the new operator is guaranteed to be real and contained in the interval $[0,\, 1]$, hence the algorithm converges a priori. We have presented numerical evidences of all these results in the numerical section. It is still an open question how to interpret from a physical point of view this new operator. We plan to tackle this problem, together with the possible acceleration of the spherical iterative filtering algorithm in a future work.

\section*{Acknowledgment}

A. Cicone is a member of the INdAM Research group GNCS and is supported in part by the Italian Ministry of the University and Research under a PRIN PNRR 2022 grant number E53D23018040001.


\begin{thebibliography}{10}
\expandafter\ifx\csname url\endcsname\relax
  \def\url#1{\texttt{#1}}\fi
\expandafter\ifx\csname urlprefix\endcsname\relax\def\urlprefix{URL }\fi
\expandafter\ifx\csname href\endcsname\relax
  \def\href#1#2{#2} \def\path#1{#1}\fi

\bibitem{huang1998empirical}
N.~E. Huang, Z.~Shen, S.~R. Long, M.~C. Wu, H.~H. Shih, Q.~Zheng, N.-C. Yen,
  C.~C. Tung, H.~H. Liu, The empirical mode decomposition and the hilbert
  spectrum for nonlinear and non-stationary time series analysis, Proceedings
  of the Royal Society of London. Series A: mathematical, physical and
  engineering sciences 454~(1971) (1998) 903--995.

\bibitem{abbasimehr2020optimized}
H.~Abbasimehr, M.~Shabani, M.~Yousefi, An optimized model using lstm network
  for demand forecasting, Computers \& industrial engineering 143 (2020)
  106435.

\bibitem{cao2019financial}
J.~Cao, Z.~Li, J.~Li, Financial time series forecasting model based on ceemdan
  and lstm, Physica A: Statistical mechanics and its applications 519 (2019)
  127--139.

\bibitem{cui2021rolling}
H.~Cui, Y.~Guan, H.~Chen, Rolling element fault diagnosis based on vmd and
  sensitivity mckd, IEEE Access 9 (2021) 120297--120308.

\bibitem{holmberg2017influence}
K.~Holmberg, A.~Erdemir, Influence of tribology on global energy consumption,
  costs and emissions, Friction 5 (2017) 263--284.

\bibitem{stetco2019machine}
A.~Stetco, F.~Dinmohammadi, X.~Zhao, V.~Robu, D.~Flynn, M.~Barnes, J.~Keane,
  G.~Nenadic, Machine learning methods for wind turbine condition monitoring: A
  review, Renewable energy 133 (2019) 620--635.

\bibitem{zeng2020review}
L.~Zeng, B.~D. Wardlow, D.~Xiang, S.~Hu, D.~Li, A review of vegetation
  phenological metrics extraction using time-series, multispectral satellite
  data, Remote Sensing of Environment 237 (2020) 111511.

\bibitem{zhang2020abrupt}
P.~Zhang, J.-H. Jeong, J.-H. Yoon, H.~Kim, S.-Y.~S. Wang, H.~W. Linderholm,
  K.~Fang, X.~Wu, D.~Chen, Abrupt shift to hotter and drier climate over inner
  east asia beyond the tipping point, Science 370~(6520) (2020) 1095--1099.

\bibitem{wu2009ensemble}
Z.~Wu, N.~E. Huang, Ensemble empirical mode decomposition: a noise-assisted
  data analysis method, Advances in adaptive data analysis 1~(01) (2009) 1--41.

\bibitem{yeh2010complementary}
J.-R. Yeh, J.-S. Shieh, N.~E. Huang, Complementary ensemble empirical mode
  decomposition: A novel noise enhanced data analysis method, Advances in
  adaptive data analysis 2~(02) (2010) 135--156.

\bibitem{torres2011complete}
M.~E. Torres, M.~A. Colominas, G.~Schlotthauer, P.~Flandrin, A complete
  ensemble empirical mode decomposition with adaptive noise, in: 2011 IEEE
  international conference on acoustics, speech and signal processing (ICASSP),
  IEEE, 2011, pp. 4144--4147.

\bibitem{zheng2014partly}
J.~Zheng, J.~Cheng, Y.~Yang, Partly ensemble empirical mode decomposition: An
  improved noise-assisted method for eliminating mode mixing, Signal Processing
  96 (2014) 362--374.

\bibitem{ur2011filter}
N.~Ur~Rehman, D.~P. Mandic, Filter bank property of multivariate empirical mode
  decomposition, IEEE transactions on signal processing 59~(5) (2011)
  2421--2426.

\bibitem{huang2009convergence}
C.~Huang, L.~Yang, Y.~Wang, Convergence of a convolution-filtering-based
  algorithm for empirical mode decomposition, Advances in Adaptive Data
  Analysis 1~(04) (2009) 561--571.

\bibitem{huang2014introduction}
N.~E. Huang, Introduction to the hilbert--huang transform and its related
  mathematical problems., Hilbert--Huang transform and its applications (2014)
  1--26.

\bibitem{dragomiretskiy2013variational}
K.~Dragomiretskiy, D.~Zosso, Variational mode decomposition, IEEE transactions
  on signal processing 62~(3) (2013) 531--544.

\bibitem{hou2011adaptive}
T.~Y. Hou, Z.~Shi, Adaptive data analysis via sparse time-frequency
  representation, Advances in Adaptive Data Analysis 3~(01n02) (2011) 1--28.

\bibitem{hou2009variant}
T.~Y. Hou, M.~P. Yan, Z.~Wu, A variant of the emd method for multi-scale data,
  Advances in Adaptive Data Analysis 1~(04) (2009) 483--516.

\bibitem{coifman2017carrier}
R.~R. Coifman, S.~Steinerberger, H.-t. Wu, Carrier frequencies, holomorphy, and
  unwinding, SIAM Journal on Mathematical Analysis 49~(6) (2017) 4838--4864.

\bibitem{yu2018geometric}
S.~Yu, J.~Ma, S.~Osher, Geometric mode decomposition., Inverse Problems \&
  Imaging 12~(4) (2018).

\bibitem{gilles2013empirical}
J.~Gilles, Empirical wavelet transform, IEEE transactions on signal processing
  61~(16) (2013) 3999--4010.

\bibitem{meignen2007new}
S.~Meignen, V.~Perrier, A new formulation for empirical mode decomposition
  based on constrained optimization, IEEE Signal Processing Letters 14~(12)
  (2007) 932--935.

\bibitem{pustelnik2012multicomponent}
N.~Pustelnik, P.~Borgnat, P.~Flandrin, A multicomponent proximal algorithm for
  empirical mode decomposition, in: 2012 Proceedings of the 20th European
  Signal Processing Conference (EUSIPCO), IEEE, 2012, pp. 1880--1884.

\bibitem{selesnick2011resonance}
I.~W. Selesnick, Resonance-based signal decomposition: A new sparsity-enabled
  signal analysis method, Signal Processing 91~(12) (2011) 2793--2809.

\bibitem{lin2009iterative}
L.~Lin, Y.~Wang, H.~Zhou, Iterative filtering as an alternative algorithm for
  empirical mode decomposition, Advances in Adaptive Data Analysis 1~(04)
  (2009) 543--560.

\bibitem{cicone2020numerical}
A.~Cicone, H.~Zhou, Numerical analysis for iterative filtering with new
  efficient implementations based on fft, Numerische Mathematik 147~(1) (2021)
  1--28.
\newblock \href {https://doi.org/https://doi.org/10.1007/s00211-020-01165-5}
  {\path{doi:https://doi.org/10.1007/s00211-020-01165-5}}.

\bibitem{cicone2016adaptive}
A.~Cicone, J.~Liu, H.~Zhou, Adaptive local iterative filtering for signal
  decomposition and instantaneous frequency analysis, Applied and Computational
  Harmonic Analysis 41~(2) (2016) 384--411.

\bibitem{barbarino2021stabilization}
G.~Barbarino, A.~Cicone, Stabilization and variations to the adaptive local
  iterative filtering algorithm: The fast resampled iterative filtering method,
  arXiv preprint arXiv:2111.02764 (2021).

\bibitem{barbarino2022conjectures}
G.~Barbarino, A.~Cicone, Conjectures on spectral properties of alif algorithm,
  Linear Algebra and its Applications 647 (2022) 127--152.

\bibitem{ghobadi2020disentangling}
H.~Ghobadi, L.~Spogli, L.~Alfonsi, C.~Cesaroni, A.~Cicone, N.~Linty, V.~Romano,
  M.~Cafaro, Disentangling ionospheric refraction and diffraction effects in
  gnss raw phase through fast iterative filtering technique, GPS Solutions
  24~(3) (2020) 85.

\bibitem{li2018entropy}
Y.~Li, X.~Wang, Z.~Liu, X.~Liang, S.~Si, The entropy algorithm and its variants
  in the fault diagnosis of rotating machinery: A review, Ieee Access 6 (2018)
  66723--66741.

\bibitem{materassi2019stepping}
M.~Materassi, M.~Piersanti, G.~Consolini, P.~Diego, G.~D'Angelo, I.~Bertello,
  A.~Cicone, Stepping into the equatorward boundary of the auroral oval:
  preliminary results of multi scale statistical analysis, Annals of Geophysics
  62~(4) (2019) GM455--GM455.

\bibitem{mitiche2018classification}
I.~Mitiche, G.~Morison, A.~Nesbitt, M.~Hughes-Narborough, B.~G. Stewart,
  P.~Boreham, Classification of partial discharge signals by combining adaptive
  local iterative filtering and entropy features, Sensors 18~(2) (2018) 406.

\bibitem{papini2020multidimensional}
E.~Papini, A.~Cicone, M.~Piersanti, L.~Franci, P.~Hellinger, S.~Landi,
  A.~Verdini, Multidimensional iterative filtering: A new approach for
  investigating plasma turbulence in numerical simulations, Journal of Plasma
  Physics 86~(5) (2020) 871860501.

\bibitem{piersanti2020inquiry}
G.~Piersanti, M.~Piersanti, A.~Cicone, P.~Canofari, M.~Di~Domizio, An inquiry
  into the structure and dynamics of crude oil price using the fast iterative
  filtering algorithm, Energy Economics 92 (2020) 104952.

\bibitem{piersanti2020magnetospheric}
M.~Piersanti, M.~Materassi, R.~Battiston, V.~Carbone, A.~Cicone, G.~D’Angelo,
  P.~Diego, P.~Ubertini, Magnetospheric--ionospheric--lithospheric coupling
  model. 1: observations during the 5 august 2018 bayan earthquake, Remote
  Sensing 12~(20) (2020) 3299.

\bibitem{sharma2017automatic}
R.~Sharma, R.~B. Pachori, A.~Upadhyay, Automatic sleep stages classification
  based on iterative filtering of electroencephalogram signals, Neural
  Computing and Applications 28 (2017) 2959--2978.

\bibitem{spogli2019role}
L.~Spogli, M.~Piersanti, C.~Cesaroni, M.~Materassi, A.~Cicone, L.~Alfonsi,
  V.~Romano, R.~G. Ezquer, Role of the external drivers in the occurrence of
  low-latitude ionospheric scintillation revealed by multi-scale analysis,
  Journal of Space Weather and Space Climate 9 (2019) A35.

\bibitem{spogli2021adaptive}
L.~Spogli, H.~Ghobadi, A.~Cicone, L.~Alfonsi, C.~Cesaroni, N.~Linty, V.~Romano,
  M.~Cafaro, Adaptive phase detrending for gnss scintillation detection: A case
  study over antarctica, IEEE Geoscience and Remote Sensing Letters 19 (2021)
  1--5.

\bibitem{yu2010modeling}
Z.-G. Yu, V.~Anh, Y.~Wang, D.~Mao, J.~Wanliss, Modeling and simulation of the
  horizontal component of the geomagnetic field by fractional stochastic
  differential equations in conjunction with empirical mode decomposition,
  Journal of Geophysical Research: Space Physics 115~(A10) (2010).

\bibitem{cicone2020study}
A.~Cicone, P.~Dell'Acqua, Study of boundary conditions in the iterative
  filtering method for the decomposition of nonstationary signals, Journal of
  Computational and Applied Mathematics 373 (2020) 112248.
\newblock \href {https://doi.org/https://doi.org/10.1016/j.cam.2019.04.028}
  {\path{doi:https://doi.org/10.1016/j.cam.2019.04.028}}.

\bibitem{stallone2020new}
A.~Stallone, A.~Cicone, M.~Materassi, New insights and best practices for the
  successful use of empirical mode decomposition, iterative filtering and
  derived algorithms, Scientific reports 10~(1) (2020) 15161.

\bibitem{rilling2007one}
G.~Rilling, P.~Flandrin, One or two frequencies? the empirical mode
  decomposition answers, IEEE transactions on signal processing 56~(1) (2007)
  85--95.

\bibitem{wu2011one}
H.-T. Wu, P.~Flandrin, I.~Daubechies, One or two frequencies? the
  synchrosqueezing answers, Advances in Adaptive Data Analysis 3~(01n02) (2011)
  29--39.

\bibitem{cicone2024one}
A.~Cicone, S.~Serra-Capizzano, H.~Zhou, One or two frequencies? the iterative
  filtering answers, Applied Mathematics and Computation 462 (2024) 128322.

\bibitem{cicone2019spectral}
A.~Cicone, C.~Garoni, S.~Serra-Capizzano, Spectral and convergence analysis of
  the discrete alif method, Linear Algebra and its Applications 580 (2019)
  62--95.

\bibitem{cicone2017multidimensional}
A.~Cicone, H.~Zhou, Multidimensional iterative filtering method for the
  decomposition of high--dimensional non--stationary signals, Numerical
  Mathematics: Theory, Methods and Applications 10~(2) (2017) 278--298.

\bibitem{rilling2007bivariate}
G.~Rilling, P.~Flandrin, P.~Gon{\c{c}}alves, J.~M. Lilly, Bivariate empirical
  mode decomposition, IEEE signal processing letters 14~(12) (2007) 936--939.

\bibitem{tanaka2007complex}
T.~Tanaka, D.~P. Mandic, Complex empirical mode decomposition, IEEE Signal
  Processing Letters 14~(2) (2007) 101--104.

\bibitem{thirumalaisamy2018fast}
M.~R. Thirumalaisamy, P.~J. Ansell, Fast and adaptive empirical mode
  decomposition for multidimensional, multivariate signals, IEEE Signal
  Processing Letters 25~(10) (2018) 1550--1554.

\bibitem{cicone2022multivariate}
A.~Cicone, E.~Pellegrino, Multivariate fast iterative filtering for the
  decomposition of nonstationary signals, IEEE Transactions on Signal
  Processing 70 (2022) 1521--1531.

\bibitem{rehman2010multivariate}
N.~Rehman, D.~P. Mandic, Multivariate empirical mode decomposition, Proceedings
  of the Royal Society A: Mathematical, Physical and Engineering Sciences
  466~(2117) (2010) 1291--1302.

\bibitem{fauchereau2008empirical}
N.~Fauchereau, G.~Pegram, S.~Sinclair, Empirical mode decomposition on the
  sphere: application to the spatial scales of surface temperature variations,
  Hydrology and Earth System Sciences 12~(3) (2008) 933--941.

\bibitem{yun2019new}
X.~Yun, B.~Huang, J.~Cheng, W.~Xu, S.~Qiao, Q.~Li, A new merge of global
  surface temperature datasets since the start of the 20th century, Earth
  System Science Data 11~(4) (2019) 1629--1643.

\bibitem{sweeney2021products}
A.~Sweeney, G.~Mungov, L.~Wright, Products and services available from us noaa
  ncei archive of water level data, in: EGU General Assembly Conference
  Abstracts, 2021, pp. EGU21--8027.

\bibitem{friis2006swarm}
E.~Friis-Christensen, H.~L{\"u}hr, G.~Hulot, Swarm: A constellation to study
  the earth’s magnetic field, Earth, planets and space 58 (2006) 351--358.

\bibitem{loto2019goes}
T.~Loto’aniu, R.~Redmon, S.~Califf, H.~Singer, W.~Rowland, S.~Macintyre,
  C.~Chastain, R.~Dence, R.~Bailey, E.~Shoemaker, et~al., The goes-16
  spacecraft science magnetometer, Space Science Reviews 215 (2019) 1--28.

\bibitem{cao2018electromagnetic}
J.~Cao, L.~Zeng, F.~Zhan, Z.~Wang, Y.~Wang, Y.~Chen, Q.~Meng, Z.~Ji, P.~Wang,
  Z.~Liu, et~al., The electromagnetic wave experiment for cses mission: Search
  coil magnetometer, Science China Technological Sciences 61 (2018) 653--658.

\bibitem{marsh1988new}
J.~Marsh, F.~Lerch, B.~Putney, D.~Christodoulidis, D.~Smith, T.~Felsentreger,
  B.~Sanchez, S.~Klosko, E.~Pavlis, T.~Martin, et~al., A new gravitational
  model for the earth from satellite tracking data: Gem-t1, Journal of
  Geophysical Research: Solid Earth 93~(B6) (1988) 6169--6215.

\bibitem{kamionkowski1999cosmic}
M.~Kamionkowski, A.~Kosowsky, The cosmic microwave background and particle
  physics, Annual Review of Nuclear and Particle Science 49~(1) (1999) 77--123.

\bibitem{hu2002cosmic}
W.~Hu, S.~Dodelson, Cosmic microwave background anisotropies, Annual Review of
  Astronomy and Astrophysics 40~(1) (2002) 171--216.

\bibitem{klosko1982spherical}
S.~Klosko, C.~Wagner, Spherical harmonic representation of the gravity field
  from dynamic satellite data, Planetary and Space Science 30~(1) (1982) 5--28.

\bibitem{pierret2013optimal}
J.~Pierret, Optimal global average annual mean temperature estimation using
  station data and spherical harmonics, Ph.D. thesis, San Diego State
  University (2013).

\bibitem{thebault2021spherical}
E.~Th{\'e}bault, G.~Hulot, B.~Langlais, P.~Vigneron, A spherical harmonic model
  of earth's lithospheric magnetic field up to degree 1050, Geophysical
  Research Letters 48~(21) (2021) e2021GL095147.

\bibitem{iso-conv}
R.~A. Kennedy, T.~A. Lamahewa, L.~Wei,
  \href{https://www.sciencedirect.com/science/article/pii/S1051200411000893}{On
  azimuthally symmetric 2-sphere convolution}, Digital Signal Processing 21~(5)
  (2011) 660--666, dASP 2009 - DEFENSE APPLICATIONS OF SIGNAL PROCESSING.
\newblock \href {https://doi.org/https://doi.org/10.1016/j.dsp.2011.05.002}
  {\path{doi:https://doi.org/10.1016/j.dsp.2011.05.002}}.
\newline\urlprefix\url{https://www.sciencedirect.com/science/article/pii/S1051200411000893}

\bibitem{GLT-bookII}
C.~Garoni, S.~Serra-Capizzano, Generalized locally toeplitz sequences: Theory
  and applications, Generalized Locally Toeplitz Sequences: Theory and
  Applications 2 (2018) 1--194.
\newblock \href {https://doi.org/10.1007/978-3-030-02233-4/COVER}
  {\path{doi:10.1007/978-3-030-02233-4/COVER}}.

\bibitem{al2014singular}
A.~Al-Fhaid, S.~Serra-Capizzano, D.~Sesana, M.~Zaka~Ullah, Singular-value (and
  eigenvalue) distribution and krylov preconditioning of sequences of sampling
  matrices approximating integral operators, Numerical Linear Algebra with
  Applications 21~(6) (2014) 722--743.

\end{thebibliography}
\end{document}